\documentclass[12pt]{amsart}
\usepackage{natbib}
\usepackage{hyperref}
\usepackage{amssymb,mathrsfs,amsmath}
\usepackage{amscd,amsmath,amsthm,amssymb,amsfonts,enumerate}
\hypersetup{
    colorlinks=true,
    linkcolor=blue,
    filecolor=magenta,
    urlcolor=cyan,
    citecolor=blue,
}

\usepackage{graphicx, color}
\usepackage{epstopdf}

\usepackage{subfig}

\def\({\left(}
\def\]{\right]}
\def\[{\left[}
\def\){\right)}

\newtheorem{thm}{Theorem}[section]

\newtheorem{asm}[thm]{Assumption}

\newtheorem{lem}[thm]{Lemma}
\newtheorem{rem}[thm]{Remark}

\newtheorem{conj}[thm]{Conjecture}
\newtheorem{defn}[thm]{Definition}

\def\ga{\gamma}

\newcommand{\bx}{\mathbf{x}}

\def\P{{\mathbb P}}
\def\E{{\mathbb E}}
\def\Cov{{\mathbb Cov}}
\def\Var{{\mathbb Var}}
\def\R{{\mathbb R}}

\newcommand{\disp}{\displaystyle}
\newcommand{\bea}{$$\begin{array}{ll}}
\newcommand{\eea}{\end{array}$$}
\newcommand{\bed}{\begin{displaymath}}
\newcommand{\eed}{\end{displaymath}}
\newcommand{\ad}{&\!\!\!\disp}
\newcommand{\aad}{&\disp}
\newcommand{\barray}{\begin{array}{ll}}
\newcommand{\earray}{\end{array}}
\newcommand{\beq}[1]{\begin{equation} \label{#1}}
\newcommand{\eeq}{\end{equation}}

\newcommand{\bedd}{\bed\begin{array}{l}}
\newcommand{\eedd}{\end{array}\eed}
\newcommand{\al}{\alpha}
\newcommand{\sg}{\sigma}

\newcommand{\e}{\varepsilon}
\newcommand{\Ga}{\Gamma}

\newcommand{\M}{{\mathcal{M}}}
\newcommand{\one}{{1}}

\newcommand{\wdt}{\widetilde}
\newcommand{\wdh}{\widehat}

\newcommand{\cd}{(\cdot)}
\newcommand{\rr}{\mathbb R}
\newcommand{\lbar}{\overline}

\newcommand{\eps}{\varepsilon}

\def\one{{\hbox{1{\kern -0.35em}1}}}

\def\({\left(}
\def\]{\right]}
\def\[{\left[}
\def\){\right)}
\def\one{{\hbox{1{\kern -0.35em}1}}}

\makeatletter \@addtoreset{equation}{section}

\def\argmax{\hbox{argmax}}

\usepackage{multicol,xcolor}

\begin{document}

\title[Optimal harvesting and stocking]{The effects of random and seasonal environmental fluctuations on optimal harvesting and stocking}

\author[A. Hening]{Alexandru Hening }
\address{Department of Mathematics\\
Texas A\&M University\\
Mailstop 3368\\
College Station, TX 77843-3368\\
United States
}
\email{ahening@tamu.edu}

\author[K. Q. Tran]{Ky Quan Tran}

\address{Department of Mathematics and Statistics\\
The
State University of New York in Korea\\
119 SongdoMoonhwa-Ro, Yeonsu-Gu\\
Yeonsu-Gu, Incheon 21985\\
Korea
}
\email{ky.tran@stonybrook.edu}

\author[S. Ungureanu]{Sergiu C. Ungureanu}
\address{Department of Economics\\
City, University of London\\
Northampton Square\\
London EC1V 0HB\\
United Kingdom}
\email{Sergiu.Ungureanu.1@city.ac.uk}

\date{}


\begin{abstract}
We analyze the harvesting and stocking of a population that is affected by random and seasonal environmental fluctuations. The main novelty comes from having three layers of environmental fluctuations. The first layer is due to the environment switching at random times between different environmental states. This is similar to having sudden environmental changes or catastrophes. The second layer is due to seasonal variation, where there is a significant change in the dynamics between seasons. Finally, the third layer is due to the constant presence of environmental stochasticity---between the seasonal or random regime switches, the species is affected by fluctuations which can be modelled by white noise. This framework is more realistic because it can capture both significant random and deterministic environmental shifts as well as small and frequent fluctuations in abiotic factors. Our framework also allows for the price or cost of harvesting to change deterministically and stochastically, something that is more realistic from an economic point of view.

The combined effects of seasonal and random fluctuations make it impossible to find the optimal harvesting-stocking strategy analytically. We get around this roadblock by developing rigorous numerical approximations and proving that they converge to the optimal harvesting-stocking strategy. We apply our methods to multiple population models and explore how prices, or costs, and  environmental fluctuations influence the optimal harvesting-stocking strategy. We show that in many situations the optimal way of harvesting and stocking is not of threshold type.

\end{abstract}

\maketitle

\setlength{\baselineskip}{0.22in}

\tableofcontents

\section{Introduction}\label{sec:int}

A fundamental problem in conservation ecology is finding the optimal strategy for harvesting a species. This problem is important because excessive harvesting can drive species extinct while under-harvesting ensures the loss of valuable resources. If one assumes that the dynamics and harvesting happen in continuous time, there has been significant progress in finding the optimal harvesting strategies which maximize the total discounted or asymptotic harvest yield---see the work by \cite{AP81, LO97, AS98, HNUW18, AH19, CHS21}. These studies have shown that in a very general setting the optimal strategy is of \textit{threshold} or \textit{bang-bang} type: there exists a threshold $w>0$ such that whenever the population is under the threshold there is no harvesting while when the population is above the threshold one harvests at the maximal (possibly infinite) rate. In this paper we look at whether this result is true in more general and realistic models.

The mathematical framework we will use is the one of stochastic differential equations with switching (SSDE). A SSDE has a discrete component that keeps track of the environment, and which changes at random times. In a fixed environmental state the system is modelled by a stochastic differential equation. This way one can capture the more realistic behaviour of two types of environmental fluctuations:
\begin{itemize}
  \item major environmental shifts (daily or seasonal changes, catastrophes),
  \item abiotic fluctuations within each environment.
\end{itemize}
We focus on the most natural setting, when the switching rates are constant. In this case, the process spends an exponentially distributed time in each environmental state and then switches to a different state. Environmental switches have been shown to fundamentally alter the fate of ecological communities by reversing competitive exclusion into coexistence or having other unexpected results \citep{HS19, HN20, BL16, B21, HNNW21}. We explore how switching impacts the harvesting and stocking of a species.

Even though the general properties for SSDE have been studied thoroughly \citep{YZ09, ZY09, NYZ17}, there are few results regarding the persistence or harvesting of ecological systems modelled by SSDE \citep{song2011optimal, Ky15, Ky17, BS16, Song16, HL20}. We fill this gap by providing an analysis complemented by an in depth look at some specific illuminating examples. In particular, we look at the logistic equation which has been used extensively in fisheries and other harvesting settings \citep{Cla10,AS98}.

The SSDE framework is generalized even further by including deterministic seasonal variation, which will make the various coefficients depend explicitly on time in a periodic fashion. There are few studies which look at the interaction of harvesting and seasonal variability. Some focus on very specific models or look only at the purely deterministic setting \citep{C88, FW98, BS03, XBD05, BM08, BSS16}. The current paper provides important generalizations to these previous results as we can analyze a very wide range of models.

In addition to modelling the ecological dynamics, one also has to have a robust way of modelling the economics.  The price of the harvested species can depend on the population size, the state of the environment and also explicitly on time. Our framework also includes a realistic cost that is incurred through stocking or harvesting. This cost is due to fees associated with harvesting-stocking policies, state constraints, certain taxes, or incentives that the manager must follow or can receive.

The price and cost functions depend on time both directly, and indirectly through the population state and the environment switching state. Economically, the indirect time dependence is important since it captures the growth-delayed-return trade-off, which is a typical feature of harvesting models. The direct time dependence adds another important layer of realism, since it can accommodate exogenous price changes. These can come from varying commodity prices (external supply and demand shocks), and varying costs of inputs of production.
\newline
The harvested quantity can indirectly change the price in the market by affecting the total supply. Because our price and cost functions depend directly on the harvested quantities, we can accommodate for this endogenous effect on the total supply. The numerical methods can simulate the price changes with a suitable choice of price and cost functions. Because the market side is not modeled explicitly, this can only be done ad hoc, for each particular application. A discussion of the importance of market variables in fisheries management can be found in \cite{sylvia1994market, pooley1987demand, asche2015economic}.

The main novelties of our work are the following:
\begin{enumerate}
\item We formulate the harvesting-stocking problem for a species that is influenced by three types of environmental fluctuations: the first due to major regime shifts, the second due to constant abiotic changes, and the third due to seasonality.
\item The price for harvesting or stocking is realistic and depends on the population size as well as the state of the environment. Furthermore, stocking and harvesting incur a cost as well.
\item We develop rigorous numerical approximation schemes based on the Markov chain approximation method.
\item We discover interesting new phenomena by analyzing in depth some important examples.
\end{enumerate}

The rest of the paper is organized as follows. In Section \ref{sec:for} we describe our model and the main results. In Section \ref{sec:ext} we discuss several extensions of the proposed model.  Particular examples are explored using the newly developed numerical schemes in Section \ref{sec:fur}. The discussion of our results is in Section \ref{sec:disc}.  Finally, all the technical proofs appear in the appendices.

\section{Model and Results}\label{sec:for}
Assume we have a probability space $(\Omega, \mathcal{F}, \P)$ satisfying the usual conditions.
Denote by $X(t)$ the size of the population at time $t\geq 0$.
For a natural population, without harvesting or stocking, the dynamics is given by
\beq{e.1} d X(t)=b\big(X(t), \al(t)\big)
dt+\sigma \big(X(t), \al(t) \big) d w(t),\eeq
where $w\cd$ is a standard Brownian motion, $\al(t)$ is an irreducible continuous-time Markov chain taking values in $\mathcal{M}=\{1, \dots, m_0\}$, and $b,\sigma:\rr_+\times \M\to \R$ are smooth enough functions. Furthermore,
we assume that the Brownian motion $w\cd$ and the Markov chain $\alpha\cd$ are independent and denote by $q_{ij}$ the transition rates of $\al(t)$. This means that
\begin{equation}
	\label{e:tran}
	\begin{array}{ll}
		&\disp \P\{\al(t+\Delta)=j~|~\al(t)=i, X(s),\al(s), s\leq t\}=q_{ij}\Delta+o(\Delta), \text{ if } i\ne j.\\
		&\disp \P\{\al(t+\Delta)=i~|~\al(t)=i, X(s),\al(s), s\leq t\}=1+q_{ii}\Delta+o(\Delta),
	\end{array}
\end{equation}
where $q_{ii}:=-\sum_{j\ne i}q_{ij}$.
As usual, we assume that $b(0, \al)=\sg(0, \al)=0$ since if the population goes extinct, it should not be able to get resurrected without external intervention (like a repopulation/stocking event). If $X(t_0)=0$ for some $t_0\ge 0$, then $X(t)=0$ for any $t\ge t_0$. Thus, $X(t)\in \R_+$ for any $t\ge 0$.

\textit{\textbf{Biological interpretation}: If at time $t$ the state of the environment is $\alpha(t)=k$ the dynamics is governed by the SDE
\begin{equation*}
d X(t)=b\big(X(t), k\big)
dt+\sigma \big(X(t), k \big) d w(t)
\end{equation*}
with nonlinear drift term $\mu(x,k)$ and diffusion term $\sigma(x,k)$. However, the environment can switch to a different state $j$ with the transition rate $q_{kj}$. In the small time $\Delta$ the probability of transition to state $j$ is approximately $q_{kj}\Delta$. }

Note that since the transition matrix $Q=(q_{ij})_{m_0\times m_0}$ is independent of the population of the species, the jump times of $\alpha(t)$ will be exponentially distributed and independent of the process $X$.

Let $U(t)$ be the harvesting-stocking rate of the population. This means that in the small time-interval $(t,t+\Delta t)$ one harvests $U(t)\Delta t$ if $U(t)>0$ and one stocks $U(t)\Delta t$ if $U(t)<0$.

The dynamics of the population $X(t)$ that includes harvesting and stocking becomes
\beq{e.2}X(t)=x+\int\limits_0^t \big( b(X(s), \al(s)) - U(s)\big) ds +  \int\limits_0^t  \sg \big(X(s), \al(s) \big) dw(s).\eeq
In order to have a well defined system we also need to impose some initial conditions:
\beq{e.3}X(0)=x\in \R_+, \quad  \al(0)=\al\in \M.\eeq
Note that $\max\{U(t),0\}$ and $\max\{-U(t), 0\}$ are the magnitudes of the harvesting and the stocking rate, respectively.
We suppose that $U(t)$ takes values in a compact subset $\mathcal{U}\subset \R$ and $0\in \mathcal{U}$. This means that the maximal harvesting and stocking rates are bounded and that one can always choose to not interfere with the population, that is, pick $U(t)=0$ at certain times $t\geq 0$.

Let $\mathcal{A}_{x, \al
}$ denote the collection of all admissible controls with initial value $(x, \al)\in  \R_+\times \M$. The collection  $\mathcal{A}_{x, \al
}$ of admissible controls is defined to contain the elements $U\cd$ such that
\begin{itemize}
	\item $U(t)$ is $\mathcal{F}(t)$-adapted,
	\item $U(t)\in \mathcal{U}$ for any $t\ge 0$,
	\item $X(t)\ge 0$ for any $t\ge 0$.
\end{itemize}

Let $P: \rr_+ \times \M\to [0, \infty)$ be the price of the species per unit. The price depends on the population size and the environmental state.
The accumulated income can be written as
$$\int_0^\infty e^{-\delta s} P\big(X(s), \al(s) \big)\cdot U(s)ds,$$
where $\delta>0$ is the time discount rate. In addition, we also suppose that the harvesting-stocking is costly, something that is described by the cost function
$C: \R_+\times \M \times \mathcal{U}\to [0, \infty)$. Thus, for a control strategy $U\cd\in \mathcal{A}_{x, \al}$, we define the \textit{performance function} as
\beq{e.4}
J\big(x, \al, U\cd \big):=  \E_{x, \al}  \int_0^{\infty} e^{-\delta s} \Big[ P(X(s), \al(s)) \cdot U(s) - C\big(X(s), \al(s), U(s)\big) \Big]ds,
\eeq
where $\E_{x, \al}$ denotes the expectation with respect to the probability law when the process $(X(t), \al(t))$ starts with initial condition $(x, \al)$. The goal is to maximize $J(x, \al, U\cd)$ and find an optimal  strategy $U^*\cd$ such that
\beq{e.5}
J\big(x,\al, U^*\cd \big)=V(x, \al):= \sup\limits_{U
\cd\in \mathcal{A}_{x, \al}} J\big(x, \al, U\cd \big).
\eeq
For notational simplicity, we collect the price function $P\cd$ and the cost function $C\cd$ into the price-cost function $p: \R_+\times \M\times \mathcal{U}\to \R$ given by
$$p(x, \al, u) = P(x, \al)\cdot u - C\big(x, \al, u\big).$$
As a result, the performance function can be written as
\beq{e.6}
J(x, \al, U\cd)=  \E_{x, \al}  \int_0^{\infty} e^{-\delta s} p\big(X(s), \al(s), U(s)\big) ds.
\eeq

The following standing assumptions are made throughout the paper.

\begin{asm}\label{a:1}
\begin{itemize}
\item[{\rm (a)}]  The functions $b(\cdot, \al)$ and $\sigma(\cdot, \al)$ are locally Lipschitz continuous for each $\al\in \M$. Moreover, for any initial condition $(x, \al)\in  \R_+\times \M$, the uncontrolled system \eqref{e.1}  has a unique global solution.
\item[{\rm (b)}]
The function $p(\cdot, \al, \cdot)$ is bounded and continuous for each $\al\in \M$.
\end{itemize}
\end{asm}
\begin{rem}
We recall that in recent papers by \cite{Ky18, Ky19} the authors assume that the stocking price is higher than the harvesting price. In the current paper, however, the stocking and harvesting have the same price which is a function of the current population size and the environmental regime. We take into account the control cost by introducing the price-cost function $p\cd$. As a result our setting is more general than those considered by \cite{Ky18, Ky19}. The cost function $C\cd$ allows us to take into account taxes, state constraints, and incentives for sustainability that arise in practical settings. We will discuss several extensions and  related formulations in Section \ref{sec:ext}. In the first extension, we look at the setting where harvesting in a short time $dt$ looks like $U(t)X(t)dt$, and $U(t)$ lives in a compact set including $0$---note that the harvesting rate in time can be unbounded in this setting. In the second extension, we look at the setting where the price is time-dependent, uncertain and its evolution is described by a stochastic differential equation. In the third extension, we consider the combined effects of random and seasonal (periodic) environmental fluctuations.
\end{rem}
A common approach in treating stochastic control problems is to characterize the value function as a viscosity solution of a Hamilton-Jacobi-Bellman equation. Formally, the associated equation of the underlying problem is given by
\beq{hjb}
\barray
\max\limits_{u\in \mathcal{U}}\left[ V'(x, \al)\big(b(x, \al) - u\big) + \dfrac{1}{2}\sg^2(x, \al)V''(x, \al) \right. \\
	 \left. + \sum\limits_{k=1}^{m_0}q_{\al k} V(x, k) + p(x, \al, u) -\delta V(x, \al)\right]=0,
\earray
\eeq
for $(x, \al)\in \rr_+\times \M$.
One can try to adopt this approach by following the techniques used by \cite{Z11}.
However, as pointed out by \cite{Z11}, under regime switching the associated Hamilton-Jacobi-Bellman equation is complicated and a closed-form solution is virtually impossible to obtain. Therefore, in order to treat the underlying formulation and extensions in Section \ref{sec:ext}, we will approximate the original problem by a sequence of discrete-time optimal control problems of approximating Markov chains.
\begin{rem} Although our work is motivated by the presence of a Markovian switching environment, the numerical schemes we developed can also be used to solve harvesting problems for diffusion models without switching. In particular, one can simply take $\M=\{1\}$ and put $q_{11}=0$.
\end{rem}

\subsection{Numerical Scheme}
\label{sec:num}
\

\noindent
We provide a numerical approach to gain information about the value function and the optimal harvesting-stocking strategy.
Using the Markov chain approximation method developed by \cite{Kushner90} and \cite{Kushner92}, we construct a controlled  Markov chain in discrete time that approximates the controlled stochastic processes. As pointed out by \cite{Kushner92} a probabilistic approach using the Markov chain approximation method for controlled diffusions has some important advantages. First, the Markov chain approximation method allows one to use physical insights derived from the dynamics of the controlled diffusion in obtaining a suitable approximation scheme. Second, the Markov chain approximation method does not require significant regularity of the controlled processes \eqref{e.2} nor does it rely on the uniqueness properties of the associated HJB equation \eqref{hjb}.

Let $h>0$ be a discretization parameter. Define $S_{h}: = \{k h: k\in \mathbb{Z}_{\ge 0}\}$ and let $\{(X^h_n, \al^h_n): n \in \mathbb{Z}_{\ge 0}\}$
be a discrete-time controlled Markov chain with state space $S_{h}\times \M$.
At each discrete-time  step $n$,  the magnitude of the harvesting-stocking component $U^h_n$ must be specified.
Let $U^h = \{U^h_n\}$  be a sequence of controls.
We denote by $q^h\((x, k), (y, l)|u\)$ the transition probability from state $(x, k)$ to another state $(y, l)$ under the control $u$.
Denote
$\mathcal{F}^h_n=\sigma\{X^h_m, \al^h_m, U^h_m, m\le n\}$.

The sequence $U^h$
is said to be admissible if it satisfies the following conditions:
\begin{itemize}
	\item[{\rm (a)}]
	$U^h_n$ is
	$\sigma\{X^h_0, \dots, X^h_{n}, \al^h_0, \dots, \al^h_n,U^h_0, \dots, U^h_{n-1}\}-\text{adapted}$ and $U^h_n\in \mathcal{U}$ for each $n$;
	\item[{\rm (b)}]  For any $(x, k)\in S_h\times \M$, we have
	\begin{equation*}
		\begin{split}
		\P \left\{ \(X^h_{n+1}, \al^h_{n+1}\) = (x, k) | \mathcal{F}^h_n\right\} &= \P \left\{ \(X^h_{n+1}, \al^h_{n+1}\) = (x, k) | X^h_n, \al^h_n, U^h_n\right\}  \\
		&= q^h \( (X^h_n, \al^h_n), (x, k)| U^h_n \);
		\end{split}
	\end{equation*}
	\item[{\rm (c)}] $X^h_n\in S_h$ for all $n\in  \mathbb{Z}_{\ge 0}$.
\end{itemize}
The class of all admissible control sequences $U^h$ for initial state $(x, \al)$ will be denoted by
$\mathcal{A}^h_{x, \al}$.

For each
 $(x, \al, u)\in S_h\times \M\times \mathcal{U}$,
we define
a family of interpolation intervals $\Delta t^h (x, \alpha, u)$. The values of $\Delta t^h (x, \al, u)$ will be specified later. Then we define
$$ t^h_0 = 0,\quad  \Delta t^h_m = \Delta t^h(X^h_m, \al^h_m, U^h_m),
\quad  t^h_n = \sum\limits_{m=0}^{n-1} \Delta t^h_m.$$
For $(x, \al)\in S_h\times \M$ and $U^h\in \mathcal{A}^h_{x, \al}$, the performance function and the value function for the controlled Markov chain is defined as
\beq{e.7}
J^h(x, \al, U^h) =  \E\sum_{m=0}^{\infty} e^{-\delta t_m^h} p(X^h_m, \al^h_m, U^h_m)  \Delta t_{m}^{h}, \quad
V^h(x, \al) = \sup\limits_{U^h\in \mathcal{A}^h_{x, \al}} J^h (x, \al, U^h).
\eeq

The corresponding dynamic programming equation for the discrete approximation is given by
\begin{equation*}
\begin{split}
V^h(x, \al) = \max_{u\in \mathcal{U}} \bigg[e^{-\delta \Delta t^h(x, \al, u)} \sum\limits_{(y, \beta)\in S_h\times \M} V^h(y, \beta)q^h\big( (x, \al), (y, \beta)| u\big)
 + p (x, \al, u) \Delta t^h(x, \al, u)\bigg].
\end{split}
\end{equation*}

We will construct the transition probabilities and interpolation intervals such that the Markov chain $\left\{(X^h_n, \al^h_n)\right\}$ approximates the process $\left\{(X\cd, \al\cd\right\}$ well, in the sense that they are locally consistent. Then the similarity between \eqref{e.6} and \eqref{e.7} suggests that the values $V^h(x, \al)$ and $V(x, \al)$ will be close for small $h$, and this will turn out to be the case. Solving the optimal harvesting problem for the chain $\left\{(X^h_n, \al^h_n)\right\}$, we obtain an approximating optimal value and an approximating optimal strategy for the continuous-time process $\big( X\cd, \al\cd\big)$. The main convergence result is given below.

\begin{thm}
\label{thm.1}
	Suppose Assumptions \ref{a:1} holds. Then for any $(x, \al) \in \R_+\times \M$,
	 $V^h(x, \al)\to V(x, \al)$ as $h\to 0$.
	Thus, for sufficiently small $h$, a near-optimal harvesting-stocking strategy of the controlled Markov chain $\{(X^h_n, \al^h_n)\}$ is also a near-optimal harvesting-stocking policy of $\big(X\cd, \al\cd\big)$.
\end{thm}

\begin{rem}
  From a control theoretic point of view, an admissible control $U\cd\in \mathcal{A}_{x, \al}$ is called \textbf{$\e$-optimal with respect to $(x, \al)$}
if $J(x, \al, U\cd)\ge V(x, \al)-\e$. A family of admissible controls $U_\e\cd\in \mathcal{A}_{x, \al}$ parameterized
by $\e>0$ is called \textbf{near-optimal} with respect to $(x, \al)$ if
$J(x, \al, U_\e\cd)\ge V(x, \al)-\psi(\e)$
 for sufficiently small $\e$, where $\psi\cd$ is a function of $\e$ satisfying $\psi(\e) \to 0$ as $\e \to 0$. The $\e$-optimal controls and near-optimal controls for the chain $\{(X^h_n, \al^h_n)\}$ can be defined similarly.
By virtue of Theorem \ref{thm.1} (see also Theorem \ref{thm:4.6}), as $h$ approaches $0$,
a near-optimal harvesting-stocking strategy of the controlled Markov chain $\{(X^h_n, \al^h_n)\}$ will provide a near-optimal harvesting-strategy for$\big(X\cd, \al\cd\big)$. As a result, the original problem can be reduced to solving for a near-optimal harvesting-stocking policy of the Markov chain $\{(X^h_n, \al^h_n)\}$.
\end{rem}

\section{Extensions}
\label{sec:ext}

\subsection{Variable effort harvesting-stocking strategies}
\label{sec:var_eff}
\

\noindent
In this subsection, we analyze an extension of our framework for which the harvesting-stocking rates are not forced to be bounded as in \eqref{e.2}---see \citet[Section 3]{Alvarez98} and work by \citet{Kha2019}. This allows us to take into account not only the impact of the harvesting-stocking effort, but also the abundance of the species itself. It is natural to assume that the harvesting or stocking rate is proportional to the population size \citep{HNUW18}. These types of harvesting strategies, where a constant fraction of the population is harvested, are called \textit{constant effort harvesting strategies} and are widely used in modeling fisheries.
We suppose that $\mathcal{U}$ is a compact subset of $\mathbb{R}$ and that $0\in \mathcal{U}$. Our harvesting-stocking strategy will be such that in the small time $dt$ we harvest or stock the amount $U(t)X(t)\,dt$ where $U(t)\in \mathcal{U}$. We call these strategies \textit{variable effort harvesting-stocking strategies}. The dynamics that includes stocking and harvesting is given by

 \beq{f.1}X(t)=x+\int\limits_0^t \big( b(X(s), \al(s)) - U(s)X(s)\big) ds +  \int\limits_0^t  \sg \big(X(s), \al(s) \big) dw(s), \quad \al(0)=\al\in \M.\eeq

 Let $\mathcal{A}_{x, \al
}$ denote the collection of all admissible controls with initial value $(x, \al)\in  \R_+\times \M$, where a strategy $U\cd$ will be in $\mathcal{A}_{x, \al
}$ if  $U(t)$ is
$\mathcal{F}(t)$-adapted, $U(t)\in \mathcal{U}$ for any $t\ge 0$, and $X(t)\ge 0$ for any $t\ge 0$.
The corresponding performance function and value function are given by
\begin{align*}
 p(x, \al, u) &= P(x, \al)\cdot x \cdot u - C\big(x, \al, u\big), \\
 \lbar {J}(x, \al, U\cd):&=  \E_{x, \al}  \int_0^{\infty} e^{-\delta s} p\big(X(s), \al(s), U(s)\big) ds,\\
 \lbar V(x, \al) &= \sup\limits_{U\cd \in \mathcal{A}_{x, \al
}} \lbar J\big(x, \al, U\cd \big).
\end{align*}

Compared with the stocking-harvesting model \eqref{e.2}, in this setting, one can harvest with higher rates when $X(t)$ is large, while we can only stock at lower rates when $X(t)$ is small. It is also clear that the stocking rate when $X(t)=0$ is $U(t)X(t)=0$ so that $V(0, \al)=0$ for any $\al\in \M$. This implies that one cannot restock a species once it is extinct.

\subsubsection{Constructing the approximation scheme.} We use the same method as in the preceding section to construct a discrete-time controlled Markov chain $\{(X^h_n, \al^h_n): n \in \mathbb{Z}_{\ge 0}\}$ with state space $S_{h}\times \M=\{k h: k\in \mathbb{Z}_{\ge 0}\}\times \M$. At each step $n$, we need to specify a control $U^{h}_n \in \mathcal{U}$.
 We need to define transition probabilities $q^h ((x, k), (y, l) | u)$ and interpolation intervals $\Delta t^h (x, \alpha, u)$
 so that the controlled Markov chain $\{(X^h_n, \al^h_n)\}$ is locally consistent with respect to the controlled diffusion \eqref{f.1}. For
 $(x, \al, u)\in S_h\times \M\times \mathcal{U}$ and
the family of the interpolation intervals $\Delta t^h (x, \alpha, u)$, we define
$$
t^h_0 = 0,\quad  \Delta t^h_m = \Delta t^h(X^h_m, \al^h_m, U^h_m),
\quad  t^h_n = \sum\limits_{m=0}^{n-1} \Delta t^h_m.$$

For $(x, \al)\in S_h\times \M$,
the class $\mathcal{A}^h_{x, \al}$ of all admissible control sequences $U^h$ for initial state $(x, \al)$ can be defined as before.
Then the performance function for the controlled Markov chain is defined as
$$
\lbar {J}^h(x, \al, U^{h}) =  \E\sum_{m=0}^{\infty} e^{-\delta t_m^h} p(X^h_m, \al^h_m, U^{h}_m)  \Delta t_{m}^{h}.
$$
The value function of the controlled Markov chain is
$$
\lbar V^h(x, \al) = \sup\limits_{U^{h} \in \mathcal{A}^h_{x, \al}} \lbar J^h (x, \al, U^{h}).
$$
The convergence result given in Theorem \ref{thm.1} is also valid in this case, if $\lbar V^h(x, \al)$ and $\lbar V(x, \al)$ are substituted for $V^h(x, \al)$ and $V(x, \al)$, respectively.

\subsection{Uncertain price functions}
\

\noindent
Since the harvesting-stocking problem we investigate is on an infinite time horizon, it is natural to look at prices that can change in time. In this section we consider a more realistic scenario with time dependent price functions \citep{Alvarez07, Kha2019}. Working with the population model given by  \eqref{e.2}, we suppose that the per-unit price-cost function is $p_0(x, \al, u) + \Phi(t)$ for $(t, x, \al, u)\in \R^2_+\times \M\times \mathcal{U},$ where $p_0\cd$ is a deterministic function and
$\Phi\cd$
satisfies the following stochastic differential equation
\beq{g.1}d\Phi(t) = b_0\big(\Phi(t), \al(t)\big)dt + \sg_0\big( \Phi(t), \al(t) \big)dw_0(t), \quad \Phi(0)=\phi\in \rr_+.\eeq
We assume that $w_0\cd$ is a standard Brownian motion in $\mathbb{R}$ which is independent of the Brownian motion $w\cd$ that drives the dynamics of the species.  We also suppose that $b_0(\cdot, \al)$ and $\sg_0(\cdot, \al)$ are locally Lipschitz continuous function for each $\al\in \M$ and are such that \eqref{g.1} has a unique solution $\Phi\cd$, with $\Phi(t)\ge 0, t\geq 0$ and $\sup\limits_{t\ge 0}\E|\Phi(t)|<\infty$.
We note that the price depends on the current population size, the environmental regime, time (through \eqref{g.1}), and is also influenced by randomness. Note that  if $\sg_0(x, \al)\equiv 0$, $\Phi\cd$ is simply time dependent and no longer uncertain. In this setting the performance and the value functions are given by
\beq{g.2}
	\begin{split}
		\wdh J(x, \al, U\cd):&=  \E_{x, \al}  \int_0^{\infty} e^{-\delta s}  \big( p_0(X(s), \al(s), U(s)) + \Phi(s) \big) U(s)  ds, \\
		\wdh V(x, \al) &= \sup\limits_{U\cd \in \mathcal{A}_{x, \al}} \wdh J\big(x, \al, U\cd \big).
	\end{split}
\eeq

\subsubsection{Constructing the approximation scheme.} In order to treat the uncertain price function given by \eqref{g.1}, we combine  $\big(X\cd, \al\cd \big)$ and $\Phi\cd$ into one controlled process $\big(\Phi\cd, X\cd, \al\cd \big)$ with initial condition $(\Phi(0), x, \al)=(\phi, x, \al)$.  The set of admissible strategies $\mathcal{A}_{x, \al}$ is defined as in Section \ref{sec:for} and does not depend on $\Phi(0)$. With the price function above,
the performance function becomes
$$
\wdh J\big(\phi, x, \al, U\cd \big):=  \E_{\phi,x, \al}  \int_0^{\infty} e^{-\delta s} \big( p_0(X(s), \al(s), U(s)) + \Phi(s) \big) U(s)ds,
$$
where $\E_{\phi, x, \al}$ denotes the expectation with respect
to the probability law when the process $(\Phi(t), X(t), \al(t))$ starts with initial condition $(\phi, x, \al)$
and the value function is
$$\wdh V(\phi, x, \al) = \sup\limits_{U\cd\in \mathcal{A}_{x, \al}} \wdh J\big(\phi, x, \al, U\cd\big).$$

To approximate the controlled process $\big(\Phi\cd, X\cd, \al\cd\big)$, we construct a discrete-time controlled Markov chain $\left\{(\Phi^h_n, X^h_n, \al^h_n): n \in \mathbb{Z}_{\ge 0}\right\}$ with state space $\wdh S_{h}\times \M$, where
$$\wdh S_{h}: = \left\{(k_1 h, k_{2} h)'\in \rr^2: k_i\in \mathbb{Z}_{\ge 0}\right\}.$$
Note that we add $\{\Phi^h_n\}$ to approximate the evolution of the price component $\Phi\cd$. Therefore, the computations are more involved than those in the preceding settings.
At each step $n$, we need to specify a control $U^h_n\in \mathcal{U}$.
We need to define transition probabilities $q^h \big((\phi, x, k), (\psi, y , l) | u \big)$ and interpolation intervals $\Delta t^h (\phi, x,  k, u)$ so that the controlled Markov chain $\left\{( \Phi^h_n, X^h_n, \al^h_n) \right\}$ is locally consistent with respect to the controlled diffusion \eqref{e.2} and  \eqref{g.1}.

 For $(\phi, x, \al)\in \wdh S_h\times \M$,
the class $\mathcal{A}^h_{x, \al}$ of all admissible control sequences $U^h$ for initial state $(\phi, x, \al)$ can be defined as before. Then the performance function for the controlled Markov chain is defined as
$$
\wdh J^h(\phi, x, \al, U^{h}) =  \E_{\phi, x, \al}\sum_{m=0}^{\infty} e^{-\delta t_m^h} \big( p_0(X^h_m, \al^h_m, U^h_m)  + \Phi^h_n\big)  U^h_m \Delta t_{m}^{h}.
$$
The value function of the controlled Markov chain is
$$\wdh V^h(\phi, x, \al) = \sup\limits_{ U^{h}\in \mathcal{A}_{x, \al} } \wdh J^h (\phi, x, \al, U^{h}).
$$
As a natural analogue to Theorem \ref{thm.1} we get the following convergence result.

\begin{thm} \label{thm.3}
	Suppose Assumption \ref{a:1} holds,  $b_0(\cdot, \al)$ and $\sg_0(\cdot, \al)$ are locally Lipschitz continuous function for any fixed $\alpha\in \M$, and that \eqref{g.1} has a unique solution $\Phi\cd$ for which $\Phi(t)\ge 0$ and $\sup\limits_{t\ge 0}\E|\Phi(t)|<\infty$. Then for any $(x, \al) \in \R_+\times \M$ and $\phi = \Phi(0)$ one has
	$$\wdh V^h(\phi, x, \al)\to \wdh V(\phi, x, \al)$$ as $h\to 0$.
	Furthermore, for sufficiently small $h$, a near-optimal harvesting-stocking strategy of the controlled Markov chain $\left\{(\Phi^h_n, X^h_n, \al^h_n)\right\}$ is also a near-optimal harvesting-stocking policy of $\big(\Phi\cd, X\cd, \al\cd \big)$.
\end{thm}

\subsection{The combined effects of seasonality and Markovian switching}
\

\noindent
In this section, we focus on another extension of \eqref{e.2} in which the population model is periodic. This is very natural if one considers that seasonal effects are periodic and strongly influence the dynamics of species.

There have been multiple papers treating periodic environments in ecology \citep{C77, C80, HC97, RMK93}. Some of these have shown how periodic forcing can create interesting new phenomena. In \cite{WH19} the authors present a synthesis on the important role of seasonality in ecology. They explain how our knowledge on seasonal dynamics is limited both empirically and theoretically. Few studies have looked at the joint effects of periodic and random fluctuations. The current section is a first step in the direction of finding the optimal harvesting-stocking strategies when both random and seasonal effects are taken into account.

If we include seasonality the dynamics is given by
\beq{h.1}X(t)=x+\int\limits_0^t \big( b(s, X(s), \al(s)) - U(s)\big) ds +  \int\limits_0^t  \sg \big(s, X(s), \al(s) \big) dw(s),\eeq
where $b(\cdot, x, \al)$, $\sg(\cdot, x, \al)$ are periodic with period $T>0$. We also suppose that the price-cost function $p(t, x, \al, u)$ is periodic with period $T$ for each $(x, \al, u)$.
The performance function becomes
\beq{h.2}
\wdt J(x, \al, U\cd):=  \E_{x, \al}  \int_0^{\infty} e^{-\delta s} p\big(s, X(s), \al(s), U(s)\big) ds,
\eeq
 and the value function is defined in the standard way. Since the functions $b\cd$, $\sg\cd$, and $p\cd$ are time-dependent and  the time horizon of the control problem is infinite, the approach we used for the previous formulations no longer works.
 In order to treat this case, we introduce a deterministic function $\Ga\cd$ to capture the time and convert it into the interval $[0, T)$. Thus, we will study the total expected discounted value starting from any time $\ga\in [0, T)$.
 Specifically, we consider a generalization of \eqref{h.1}--\eqref{h.2} given by
 \beq{h.3}
	\begin{split}
		\ad X(t)=x+\int\limits_\ga^t \big( b(\Ga(s), X(s), \al(s)) - U(s)\big) ds +  \int\limits_\ga^t  \sg(\Ga(s), X(s), \al(s)) dw(s),\\
		\ad \Gamma(t) = t-mT, \quad \text{ where } m \in \mathbb{Z}_{\geq 0} \text{ such that } t \in [mT,mT+T) .	
	\end{split}
 \eeq

 We consider the combined process $\big(\Ga\cd, X\cd, \al\cd\big)$.  The performance function  and the value function are given by
 \beq{h.4}
	\begin{split}
		\wdt J(\ga, x, \al, U\cd):&=  \E_{\ga, x, \al}  \int_\gamma^{\infty} e^{-\delta s} p\big(\Ga(s), X(s), \al(s), U(s)\big) ds,\\
		\wdt V(\ga, x, \al) &= \sup\limits_{U\cd \in \mathcal{A}_{\ga, x, \al}} \wdt J\big(\ga, x, \al, U\cd\big),
	\end{split}
\eeq
where $\E_{\ga, x, \al}$ denotes the expectation with respect
to the probability law of the process $\big(\Ga(t), X(t), \al(t)\big)$ having initial conditions $\big(\Ga(\ga), X(\ga), \al(\ga)\big)=(\ga, x, \al)$.

\subsubsection{Constructing the approximation scheme.} In this case, we need to use two positive parameters $h_1$ and $h_2$, where $T$ is a multiple of $h_1$. Let $h=(h_1, h_2)$.
We construct a discrete-time controlled Markov chain $\left\{(\Ga^h_n, X^h_n, \al^h_n): n \in \mathbb{Z}_{\ge 0}\right\}$ with state space $\wdt S_{h}\times \M$, where
$$\wdt S_{h}: = \left\{(\ga, x)=(k_1 h_1, k_2 h_2)'\in \rr^2: k_i\in \mathbb{Z}_{\ge 0}, k_1\le T/h_1\right\}.$$
Note that we add $\{\Ga^h_n\}$ to approximate the time $\Ga\cd$. At each step $n$, we need to specify a control $U^h_n\in \mathcal{U}$.
We need to define transition probabilities $q^h \big((\ga, x, k), (\lambda, y , l) | u\big)$ and interpolation intervals $\Delta t(\ga, x, k, u)$ so that the controlled Markov chain $\left\{(\Ga^h_n, X^h_n, \al^h_n)\right\}$ is locally consistent with respect to the controlled diffusion \eqref{h.3}.

For $(\ga, x, \al)\in \wdt S_h\times \M$,
the class $\mathcal{A}^h_{\ga, x, \al}$ of all admissible control sequences $U^h$ for initial state $(\ga, x, \al)$ can be defined as before. Then the performance function for the controlled Markov chain is defined as
$$
\wdt J^h(\ga, x, \al, U^{h}) =  \E\sum_{m=0}^{\infty} e^{-\delta t_m^h} p(\Ga^h_n, X^h_m, \al^h_m, U^{h}_m) \Delta t_{m}^{h}.
$$
The value function of the controlled Markov chain is
$$
\wdt V^h(\ga, x, \al) = \sup\limits_{  U^{h}\in \mathcal{A}_{\ga, x, \al} } \wdt J^h (\ga, x, \al, U^{h}).
$$
It should be noted from the periodicity of the problem that for any $(\ga, x, \al)\in [0, T) \times \rr_+ \times \M$ one has
$$\wdt V(\ga, x, \al) = e^{-\delta T}\wdt V(\ga+T, x, \al), \quad \wdt V^h(\ga, x, \al) = e^{-\delta T}\wdt V^h(\ga+T, x, \al).$$
This property will allow us to work with the compact time interval $[0, T]$ instead of the entire infinite horizon $[0, \infty).$
The convergence result is given below.

\begin{thm} \label{thm.4}
Suppose the following assumptions hold:
\begin{enumerate}
  \item The functions $b(\cdot, \cdot, \al)$ and $\sigma(\cdot, \cdot, \al)$ are locally Lipschitz continuous for each $\al\in \M$.
  \item For any initial condition $(\ga, x, \al)\in  [0, T) \times \R_+ \times \M$, the uncontrolled system \eqref{h.1}  has a unique global solution.	
  \item The price $p(\cdot, \cdot, \al, \cdot)$ is bounded and continuous for all $\al\in \M$.
  \item For each $(x, \al)\in \rr_+\times \M$ and $u\in \mathcal{U}$, the functions $b(\cdot, x, \al)$, $\sg(\cdot, x, \al)$, $p(\cdot, x, \al, u)$ are periodic with period $T>0$.
\end{enumerate}
 Then for any $(\ga, x, \al) \in [0, T)\times \R_+\times \M$, we have
	 $$\wdt V^h(\ga, x, \al)\to \wdt V(\ga, x, \al)$$ as $h\to 0$.
	Furthermore, for sufficiently small $h$, a near-optimal harvesting-stocking strategy of the controlled Markov chain $\left\{(\Ga^h_n,X^h_n, \al^h_n)\right\}$ is also a near-optimal harvesting-stocking policy of $\big(\Ga\cd, X\cd, \al\cd\big)$.
\end{thm}


\section{Numerical Examples}\label{sec:fur}
We explore some numerical examples, using various models and assumptions.
Throughout this section, we suppose the time discount rate is $\delta = 0.02$. Let $B\in (0, \infty)$ be an upper bound introduced for computational purposes. If not specified, we will take $B=4$ in our examples.
We assume that the environment only switches between two states, so $\M= \{1, 2\}$. We exhibit various price-cost formulations to explore for features, and possibly new phenomena, in the various considered models.

\subsection{Introducing switching in a Verhulst-Pearl system}\label{sec:intro_lv}
\

\noindent
A model that is extensively used in biology is the Verhulst-Pearl one. The dynamics that includes switching is given by

$$dX(t)= X(t)\big(\mu(\alpha(t)) - \kappa(\alpha(t)) X(t) \big)\,dt + \sigma \big(\alpha(t)\big) X(t)dw(t).$$
Here $\mu(\alpha)$ is the growth rate in environment $\alpha$, $\kappa(\alpha)$ is the intraspecific competition rate in environment $\alpha$, and $\sigma^2(\alpha)$ is the variance of the per-capita environmental fluctuations. Suppose there are $m_0$ environmental states in environment $\alpha\cd$.
One can fix an environment $\alpha$ and look at the dynamics of the population in that environment
$$dX^\al(t)= X^\al(t)\big(\mu(\alpha) - \kappa(\alpha) X^\al(t)\big)\,dt + \sigma(\alpha) X^\al(t)dw(t).$$
The stochastic growth rate of the species \citep{CE89,C00,SBA11,HN18} in environment $\al$ is given by
$$
r(\alpha) = \mu(\al)-\frac{\sigma^2(\al)}{2}.
$$
Following \cite{EHS15}, this stochastic growth rate completely determines the long term behavior of the population in each fixed environment. If $r(\alpha)>0$ we get the convergence to a unique stationary distribution on $(0,\infty)$. If $r(\alpha)<0$ the population goes extinct in environment $\alpha$ and with probability $1$ we have
$$
\lim_{t\to \infty} \frac{\ln X^\al(t)}{t} = r(\al)<0.
$$
If $r(\al)=0$, by \cite{EHS15} the population process is null recurrent in environment $\al$. This means the population does not go extinct but also does not converge to a stationary distribution; it keeps fluctuating between large and small values.
Let $(\nu_1,\dots,\nu_{m_0})$ be the stationary distribution of the Markov chain $\alpha(t)$. If
$$r = \sum_{k=1}^{m_0}\nu_k\left(\mu(k)-\frac{\sigma^2(k)}{2}\right)>0,$$
we get by \cite{HL20} that the population converges, if there is switching, to a stationary distribution.

We consider a Verhulst population model with $\mu( \alpha)= 4-\alpha, \; \kappa(\alpha)=2, \; \sigma(\alpha)= 1 , \; (x, \alpha)\in  \R_+\times \{1, 2\}.$
As a result the stochastic growth rate in environment $\alpha$ is
\begin{equation*}
	r(\alpha) = (4-\alpha) - \frac{1}{2},
\end{equation*}
which implies $r(1)>0$ and $r(2)>0$, so if the environment would be fixed to either of the $\alpha$ values, the species would converge to a unique stationary distribution on $\R_+$. Suppose that the generator $Q$ of the Markov chain $\al\cd$ is given by $q_{11} = -0.1$, $q_{12} = 0.1$, $q_{21}=0.1$, $q_{22} = -0.1$. This implies that the stationary distribution of $\alpha(t)$ is $(\nu_1,\nu_2) = (0.5, 0.5)$. Henceforth, let the set of controls be $\mathcal{U}=\{u :u = k/500, k\in \mathbb{Z}, -1000\le k\le 1500\}$ where unspecified.

In a first experiment, we compare this model to a baseline model, where $\mu(\al) = 2.5$, which is in between the two values for the switching environments.
To isolate the effect of switching, in our first example we keep the price and cost functions simple: $P(x,\alpha) = 1$ and $C(x,\alpha,u) = 0$.

We have presented the formulation theoretically in Section \ref{sec:for}, and we describe a detailed procedure in Section \ref{sec:num} and Appendix \ref{sec:ap_a}.
In the following numerical formulations, we will apply similar procedures, described in the same sections.
For an admissible strategy $U\cd$ we have
\begin{equation*}
	J\big(x, \al, U\cd \big)=\E\int_0^\infty e^{-\delta s} p\big(X(s), \al(s), U(s)\big) ds.
\end{equation*}
Based on the algorithm constructed above and in the Appendixes, we carry out the computation by using the methods in \cite{Kushner92}.
We take the initial control $U_0(x, \al)\equiv0$ and set the initial values $V^h_0(x, \al)\equiv 0$.
We outline how to find the sequence of values of $V_n^h(\cdot)$ as follows.
At each level $x=h,2h, \dots, B$, $\al\in \M$, and control $u\in \mathcal{U}$, we compute
$$V_{n+1}^h(x, \al \,|\, u) =e^{-\delta \Delta t^h(x, \al, u) }\sum\limits_{(y, \beta)\in S_h \times \M} V^{h}_n (y, \beta
) q^h \big((x, \al), (y, \beta) | u \big) + p(x, \al, u)\Delta t^h(x, \al, u).$$
Working with the compact set state space $[0, B]$, we use reflection if $x=B$; that is, $V^h_n(B+h, \beta) = V^h_n(B, \beta)$ for any $\beta\in \M$.
Then we choose the control $U^h_{n+1}(x, \al)$ and record an improved value
$V^h_{n+1}(x, \al)$ by
$$U^h_{n+1}(x, \al) = \argmax_{u\in \mathcal{U}} V_{n+1}^h(x, \al \,|\, u),  \quad V^h_{n+1} (x, \al) = V^h_{n+1}(x, \al\,|\, U^h_{n+1}(x, \al)).$$
The iterations stop as soon as the increment $V^h_{n+1}\cd-V^h_n\cd$ reaches some tolerance level. We set the error tolerance to be $10^{-8}$.

\begin{figure}[h!tb]
	\begin{center}
	\subfloat{{
		\includegraphics[scale=0.6]{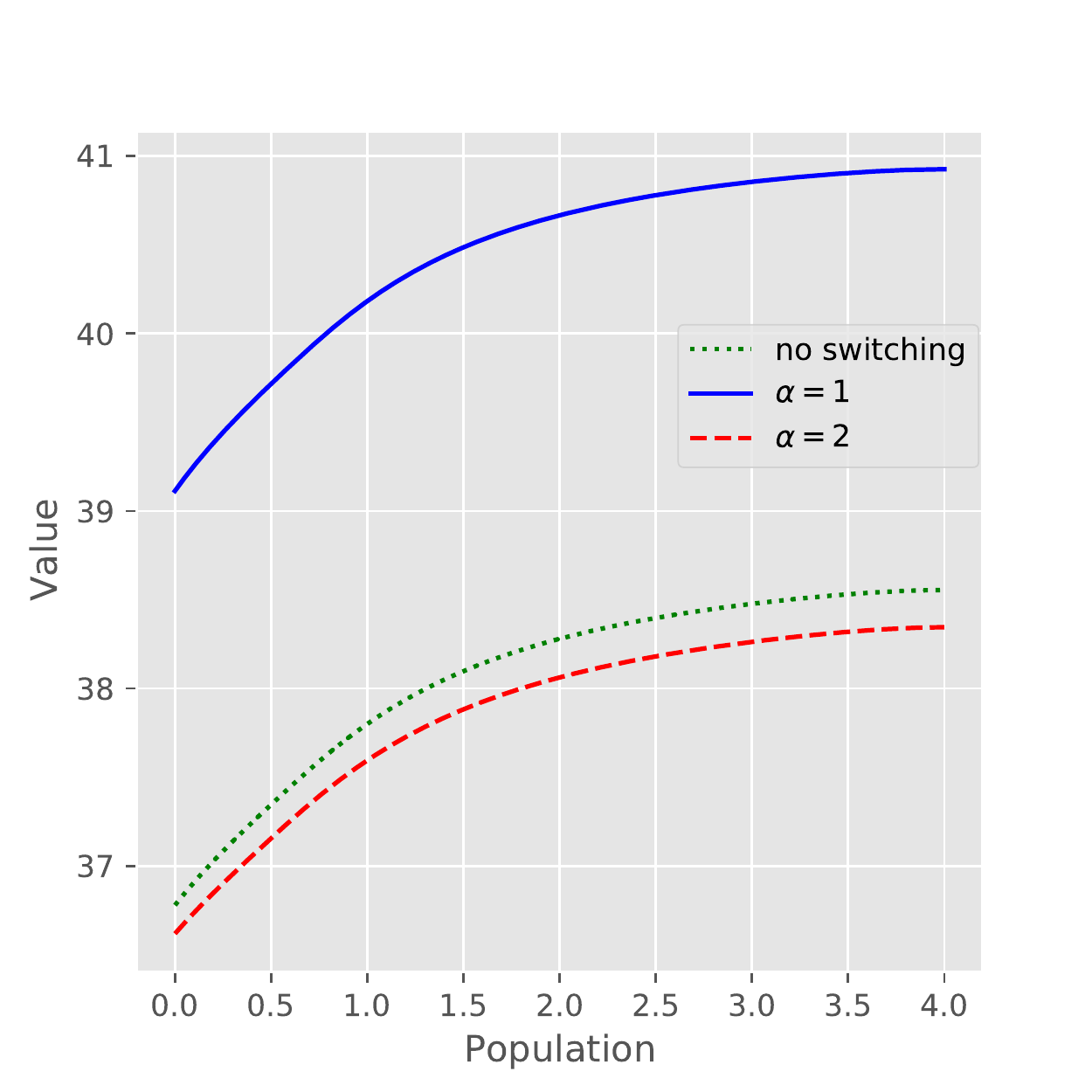}
	}}
	\subfloat{{
		\includegraphics[scale=0.6]{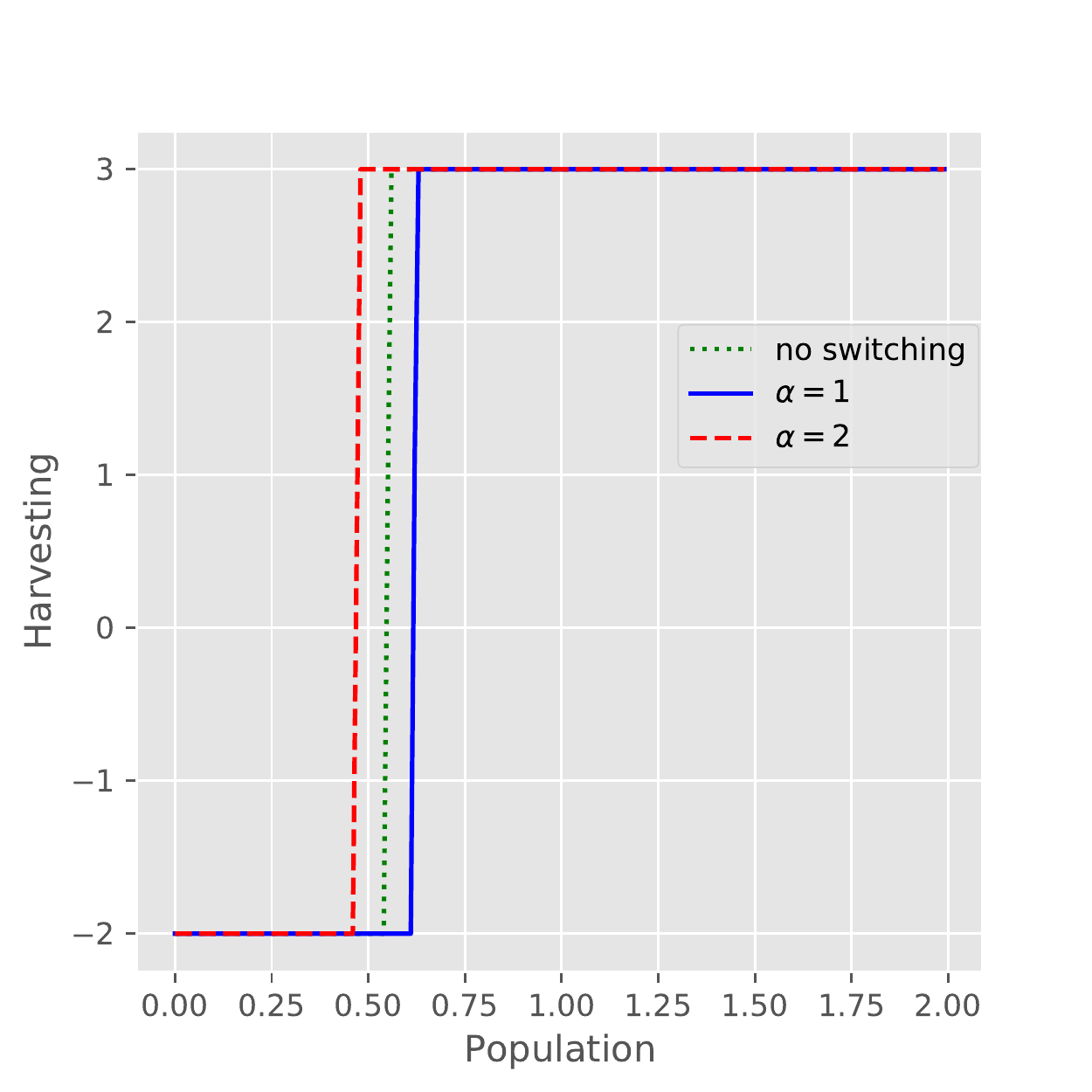}
	}}
	\caption{Value function (left) and optimal harvesting-stocking rate (right) for a model with switching affecting $\mu(\al) = 4 - \al$, compared with a baseline model with no switching and $\mu(\al) = 2.5.$}
	\label{fig:switching_1}
	\end{center}
\end{figure}

The numerical result is shown in Figure \ref{fig:switching_1}. The value function in the left panel is increasing in both the switching and the baseline models, as expected. Moreover, it is concave, which implies that the marginal value of one small unit of population decreases. This is expected because population growth becomes less favourable as the population size increases---a result of intra-species competition.

Because the transition probability rates were chosen to be small, the value function of the baseline model is in-between the value functions for the two $\al$ states of the switching model. When the transition probability rates increase, that will not always be the case, because the value function in one state is determined, in part, by the value function in the other state.

Another expected result is that the optimal harvesting rate in the second panel of Figure \ref{fig:switching_1} is bang-bang---it switches between extremes at a certain population threshold. This should be the case, because the cost of harvesting-stocking is 0, and therefore independent of the value. This was shown theoretically for a model without switching in \citet{HNUW18}. In the second panel, we also see that the optimal stocking-harvesting transition happens at a lower population level in the growth-unfavorable state $\al=2$, when compared with the intermediate growth baseline and the growth-favorable state $\al=1$. Intuitively, lower growth prospects imply optimal extraction should start at a lower population level.


Numerical experiments where $\sigma(\al)$ and $\kappa(\al)$ are also allowed to depend on $\al$ show results with similar features, so we omit them for brevity.

\subsection{Analysis of the effect of the control cost}\label{sec:cost_num}
\

\noindent
We expect to find all or nothing harvesting-stocking when the cost is not convex in the rate, even when there is environmental switching. Using the same set-up as above, but now with a convex cost function $C(x, \al, u)=u^2/2$, we find that the optimal control is not bang-bang anymore (Figure \ref{fig:switching_cost_1}). Further numerical experiments (Figures \ref{fig:cost_exp_1}, \ref{fig:cost_exp_3}, Appendix \ref{sec:numerical}) support the following conjecture, which says that, with cost functions that are not convex, the optimal strategies are bang-bang, with thresholds that may depend on the states $\alpha$.
\begin{conj}
Suppose we have one species that evolves according to \eqref{e.2} and suppose Assumption \ref{a:1} holds, the price $P$ only varies with the environment, the cost does not depend on the population size,  $C(x, \al, u)=C(\al,u)$ and the cost is not convex in the harvesting rate $u$. Furthermore, assume that in the absence of harvesting the species persists, i.e.,
$$\lambda(\mu_0) = \sum_{k=1}^{m_0}\nu_k\left(\mu(k)-\frac{\sigma^2(k)}{2}\right)>0.$$
One can construct the optimal harvesting strategy $U^*$ as follows. There exists a threshold $0\leq u^*(k)\leq \infty$ for the population such that, if $X(0-)=x$, then
\begin{equation}\label{e:cost_conj}
	\begin{aligned}
		U^*(t)
		&= \begin{cases}
			\inf\; \mathcal{U}, & \mbox{\text{ if } $t>0,\; \alpha(t)=k,\; X(t) < u^*(k)$},\\
			\sup\; \mathcal{U}, & \mbox{\text{ if } $t>0,\; \alpha(t)=k,\; X(t)\geq u^*(k)$}.
		\end{cases}
	\end{aligned}
\end{equation}
\end{conj}

\begin{figure}[h!tb]
	\begin{center}
		\subfloat{{
			\includegraphics[scale=0.6]{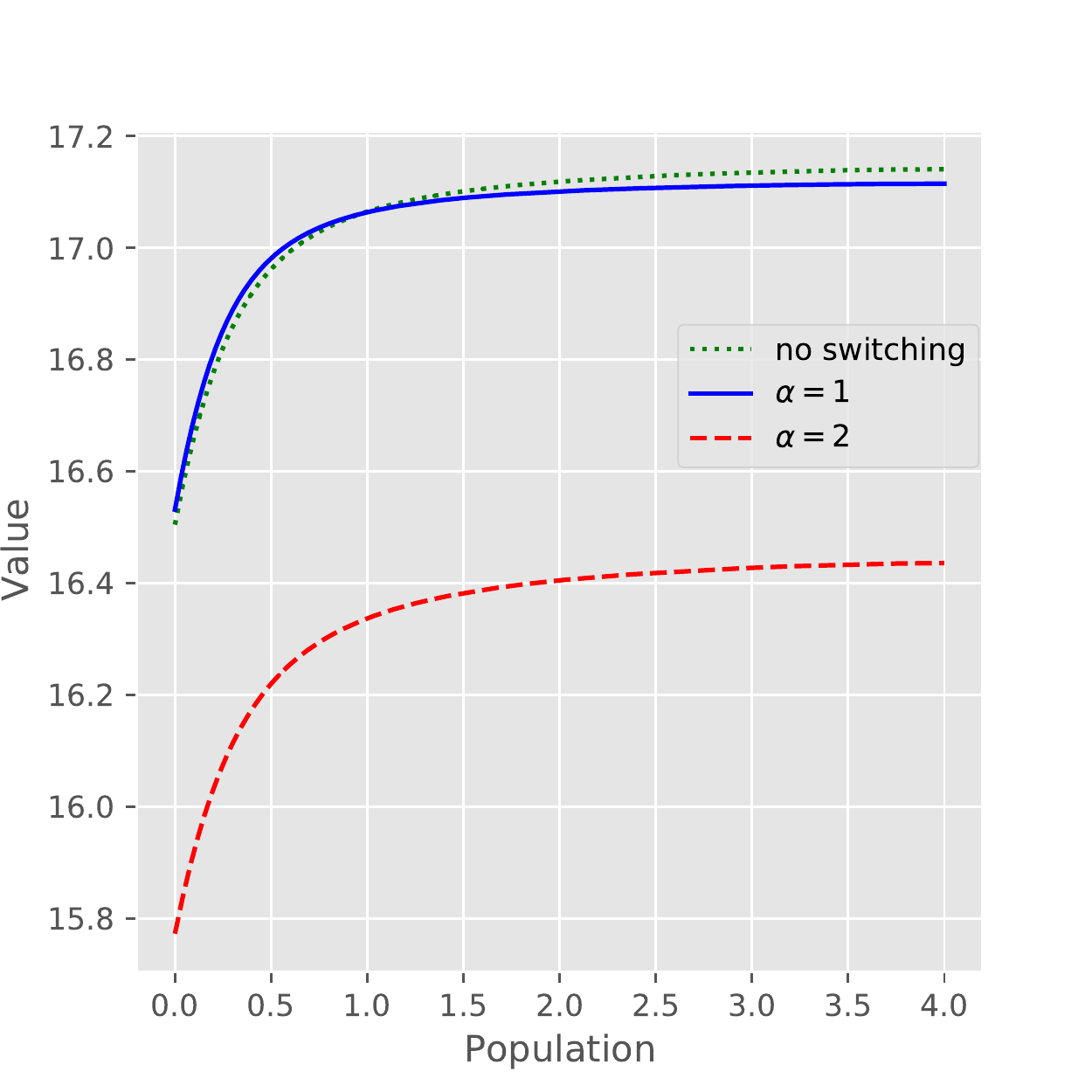}
		}}
		\subfloat{{
			\includegraphics[scale=0.6]{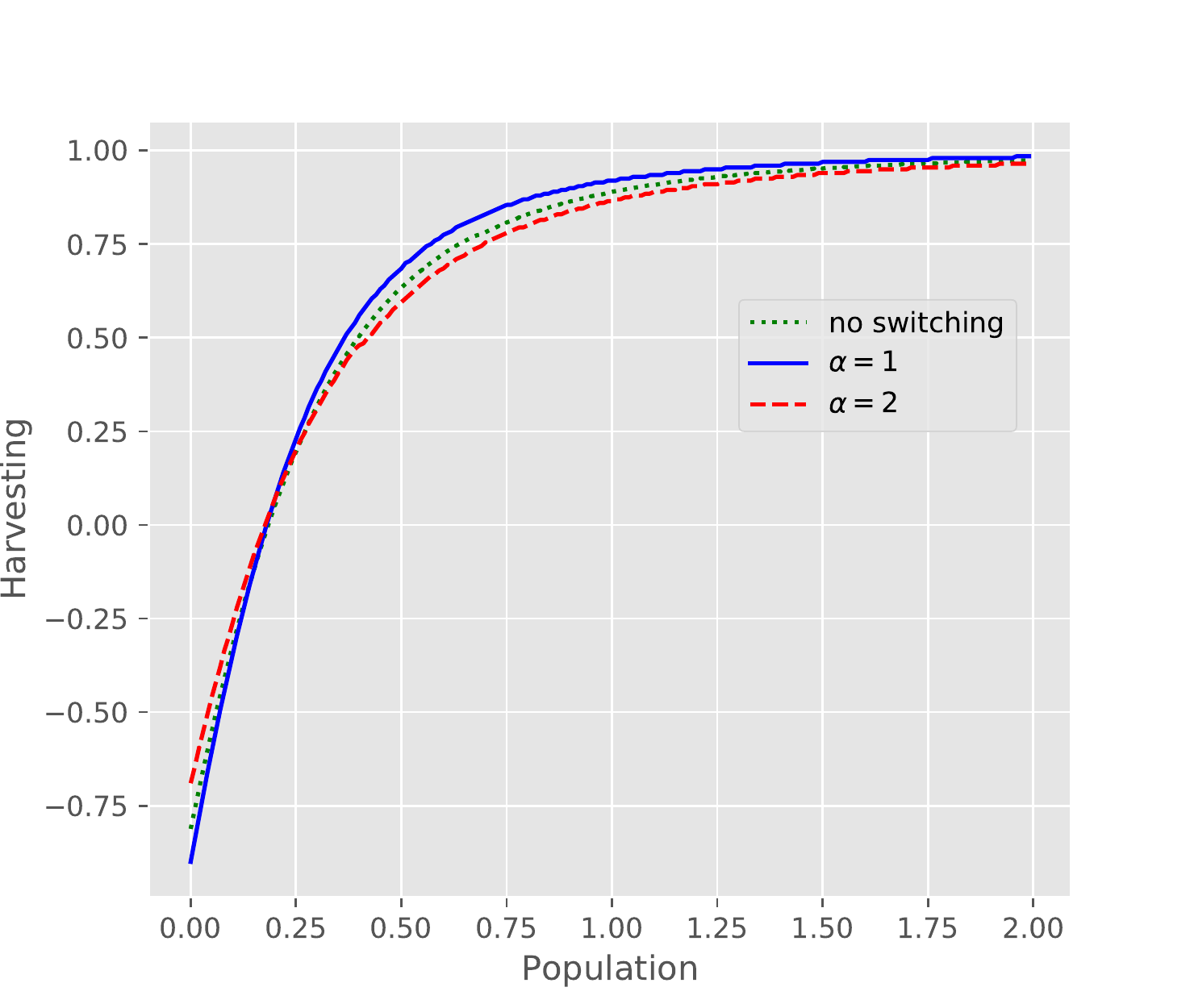}
		}}
		\caption{Value function (left) and optimal harvesting-stocking rate (right) for a model with switching affecting $\mu(\al) = 4 - \al$, and convex cost $C(x,\al,u) = u^2/2$, compared with a baseline model with no switching, $\mu(\al) = 2.5$, and the same cost function.} \label{fig:switching_cost_1}
	\end{center}
\end{figure}

From the right panel of Figure \ref{fig:switching_cost_1}, we note that the growth favorable state, $\alpha = 1$, is more conducive to expensive harvesting and stocking, which is intuitive and expected. The value functions are again concave and increasing in population, as expected. This will be common across all numerical examples considered. Henceforth we may omit the value function from figures.

\subsection{Large and small transition rates with switching}\label{sec:transition_rates}
\

\noindent
In the left panel of Figure \ref{fig:switching_cost_1}, the apparent overlap between the value function of the baseline model and the switching model in the state $\al = 1$ is coincidental. Numerical experiments show that increasing the switching probability rates pulls the value function graphs for $\al=1,2$ toward each other. If the switching rates were very small, the green baseline value function would be sandwiched between the red and blue value functions. As the switching rates become large, the value functions of the two switching states will overlap in the limit.\footnote{The value function in a switching state is influenced by the value function in the other state, since there is a probability of transition.} Figure \ref{fig:val_q_conv} shows this process for the same specification as above, when $q_{12} = q_{21} \in \{0,0.01, 0.1, 1, 10, \infty\}$. The following remark confirms the numerical experiments that suggest convergence.

\begin{figure}[h!tb]
	\begin{center}
		\subfloat{{
			\includegraphics[scale=0.6]{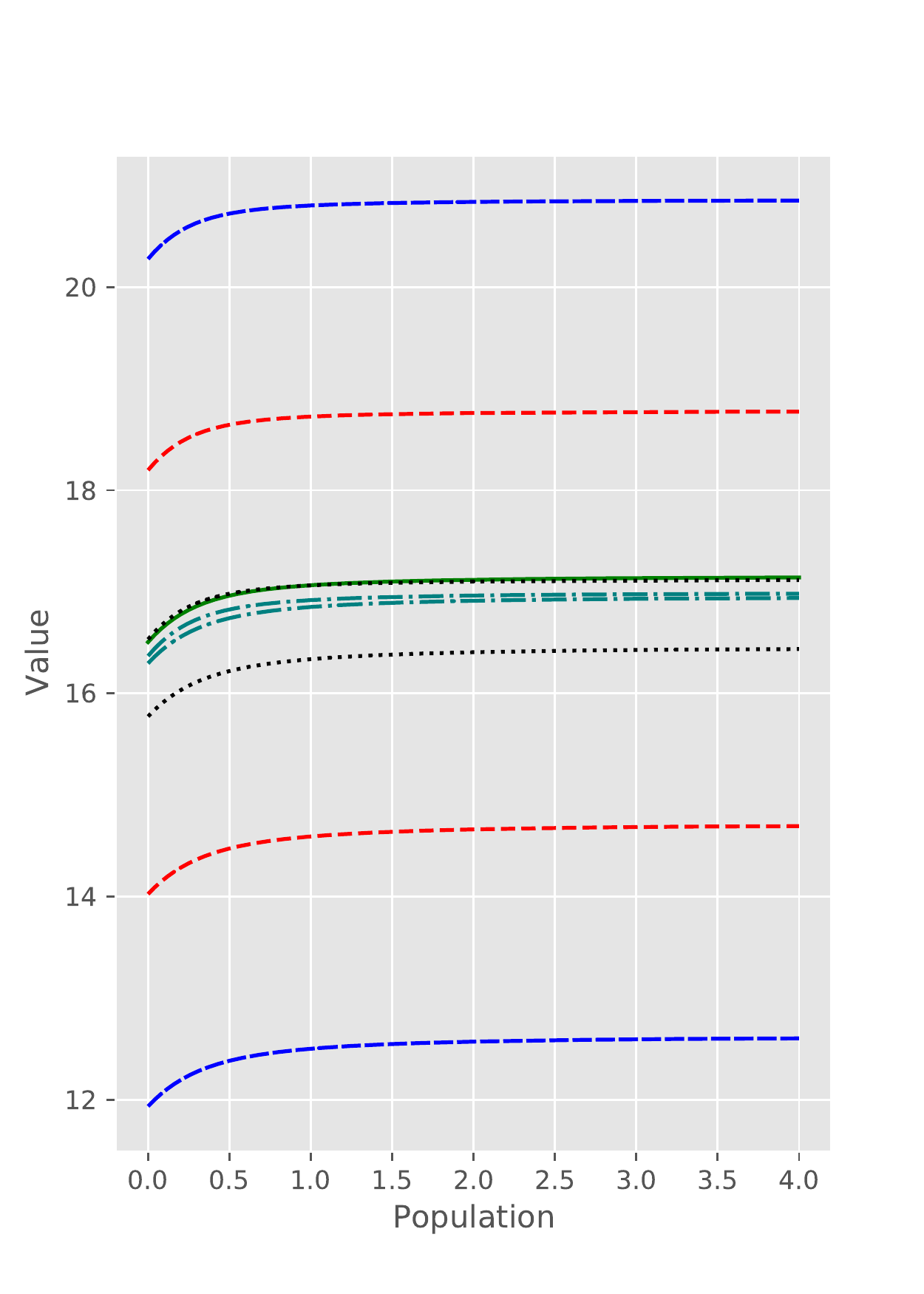}
		}}
		\subfloat{{
			\includegraphics[scale=0.6]{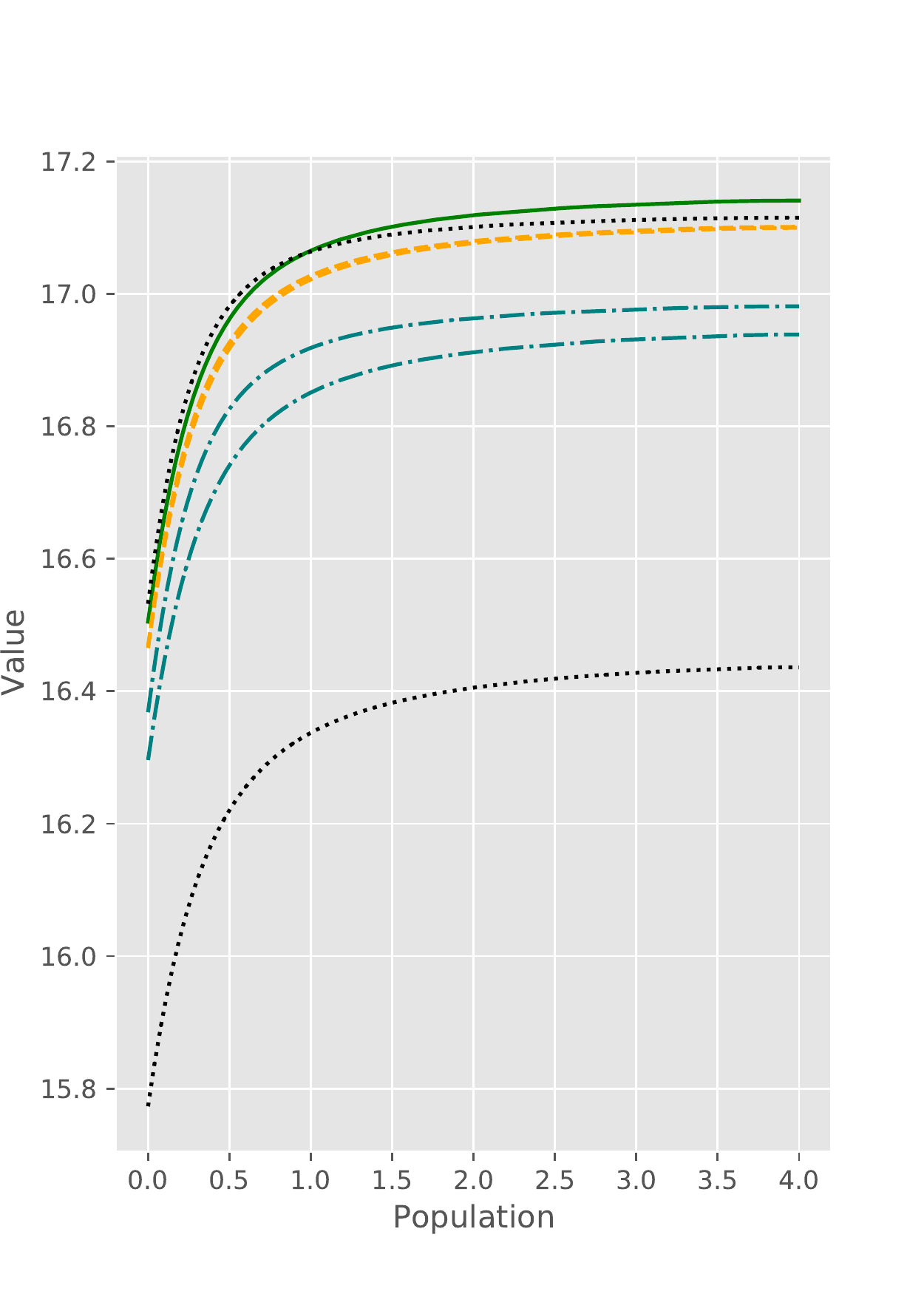}
		}}
		\caption{The left and right panels show the same value function graph, but on different vertical scales. The model is discussed in Sections \ref{sec:cost_num}, \ref{sec:transition_rates}. The green solid curves in both panels corresponds to a model without switching and $\mu = 2.5$, or to a model with $\mu = 4 - \al$ and infinitely large switching rates. The blue long-dashed curves in the left panel show models with no switching and $\mu = 3$, $\mu=2$ respectively, or equivalently a model with $\mu = 4-\al$ and infinitely slow switching rates. The red dashed curves in the left panel are for a model with $\mu = 4-\al$ and $q_{12} = q_{21} =  0.01$. The black dotted curves in both panels correspond to the same model, but $q_{12}=q_{21}=0.1$. The teal dot-dashed curves in the right panel correspond to $q_{12}=q_{21}=1$, and the orange dashed curves in the right panel to $q_{12}=q_{21}=10$.} \label{fig:val_q_conv}
	\end{center}
\end{figure}

\begin{rem}
	If we divide the switching rates by $\eps$ and let $\eps\downarrow 0$ we have a slow-fast dynamics where the switching happens on a much shorter time scale. One can show that the dynamics without harvesting \eqref{e.1} will converge \citep{HL20} to \beq{e:mix} d \bar \xi(t)=\bar b(\bar \xi(t))
	dt+\bar \sigma(\xi(t)) d w(t),\eeq
	where
	\begin{equation}\label{e:fbar}
		\bar b(\bx) = \sum_{k=1}^{m_0} b(k)\nu_k.
	\end{equation}
	and
	\begin{equation}\label{e:gbar}
		\bar \sigma = \sqrt{\sum_{k=1}^{m_0}\sigma^2(k)\nu_k}.
	\end{equation}
	Here $(\nu_1,\dots,\nu_{m_0})$ is the stationary distribution of the Markov chain $\alpha(t)$.
	We conjecture that the optimal stocking-harvesting strategies $U^{\eps,*}$ of the system \eqref{e.2} converge as $\eps\downarrow 0$ to the optimal harvesting-stocking strategy $\bar U^*$ of the system
	\beq{e.2bar}\bar X(t)=x+\int\limits_0^t \big( \bar b(\bar X(s), \al(s)) - \bar U(s)\big) ds +  \int\limits_0^t  \sg(\bar X(s), \al(s)) dw(s).\eeq
\end{rem}

Further numerical experiments, which we omit, show that if the switching rates are such that they favour one state over the other in the stationary distribution, the value function of a simplified system with only the favoured state will be closer to the value function with switching, than the value function of a simplified system with only the unfavoured state. That is because the environment is mostly in the favored state, and its conditions are therefore more relevant to the valuation.

\subsection{Switching in the cost and price functions}
\

\noindent
We have shown how the dependence of the cost function on the harvesting rate determines whether harvesting is bang-bang.
It turns out that non-trivial price functions may also smooth out the optimal harvesting rate.
We have looked at three typical functional forms for a per-unit price-cost function $p(x, \al, u)=p_0(u)$, two of which can capture monopolistic competition, i.e., imperfect competition.
The first is a simple constant price function, $p_0(u) = p,\; p>0,$ which is appropriate when the market influence of the extraction operation is low. The second has a piecewise linear decreasing form given by
\begin{equation}
	p_0(u) =	
	\begin{cases}	
		\overline{p}, & \kappa_1 - \kappa_2 u \ge \overline{p}, \\
		(\kappa_1 - \kappa_2 u),  \quad 0 < &\kappa_1 - \kappa_2 u < \overline{p}, \\
		0, &\text{else}.
	\end{cases}
\end{equation}
The third has a constant price elasticity of demand form, and is given by 
\begin{equation}
	p_0(u) =	
	\begin{cases}	
		\overline{p},&\kappa_1 |\kappa_2 + u|^{1/\epsilon} \ge \overline{p}, \\
		\kappa_1 |\kappa_2 + u|^{1/\epsilon},  \quad 0 \le &\kappa_1 |\kappa_2 + u|^{1/\epsilon} < \overline{p}. \\
	\end{cases}
\end{equation}
The functions represent instantaneous market demand, and obey common sense assumptions.
In particular, the price decreases as the supplied quantity increases, and the price is always positive.
Both forms are typical choices in modeling market behavior.
The economic model implied by the price dependence is the following: the instantaneous harvested rate $u$ is a quantity placed in a common market, where the total supplied by all sources is what determines the instantaneous common final price $u$, and hence the returns.
\newline
\indent
The constants $\kappa_1, \kappa_2$ are positive. They capture, among other things, the contribution to total harvested quantities from different sources. In the third formulation, $\epsilon$ is negative and represents the price elasticity. The constant $\overline{p}$ is added to account for unbounded negative flow rates (e.g., unrealistically large stocking events) in the economic demand model. Large negative values of $u$ could make the overall market supply negative---something evidently not possible---so we ignore such events for practical applications. That is, we want to keep prices below $\overline{p}$ in our examples.

Figure \ref{fig:price_var_ctrl} shows, as an example, the harvesting strategy corresponding to two price functions, also known as demand functions, based on the functional forms discussed.
The left panel corresponds to linear demand, $p(u) = 1 - u/4$, and the right panel to $p(u) = (1+u/3)^{-1}$.
The second form is for a good with unit elasticity.\footnote{The \textit{(price) elasticity} of a demand function is a measure of the responsiveness of the quantity demanded to price changes, measured in an unit free manner, defined as $\epsilon = \frac{du}{dp} \cdot \frac{p}{u}$, and widely used in economics. $\epsilon = -1$ is defined as ``unit elasticity."}
The constants were chosen so that price will not be negative in the harvesting range, but are otherwise arbitrary.
Different choices of slopes and elasticity values lead to similar qualitative conclusions---with no cost, when $u p(u)$ is concave, we have a harvesting strategy that is continuous in the population level.

\begin{figure}[h!tb]
	\begin{center}
		\subfloat{{
			\includegraphics[scale=0.6]{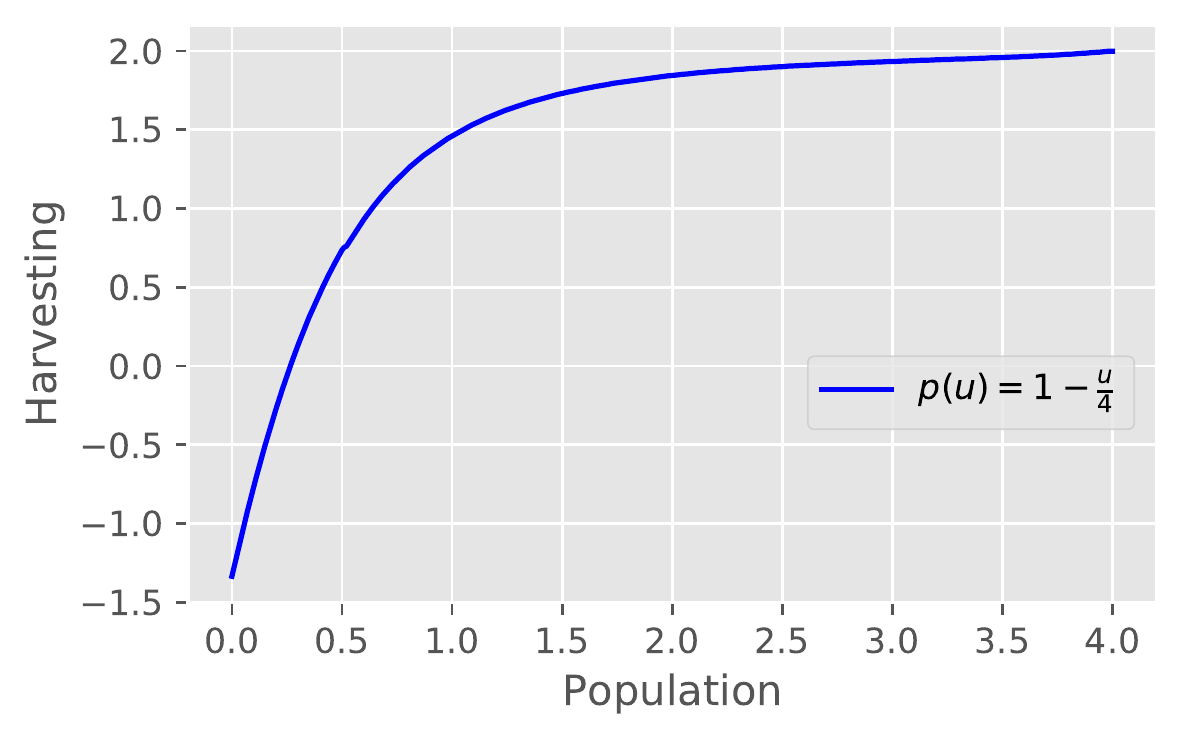}
		}}
		\subfloat{{
			\includegraphics[scale=0.6]{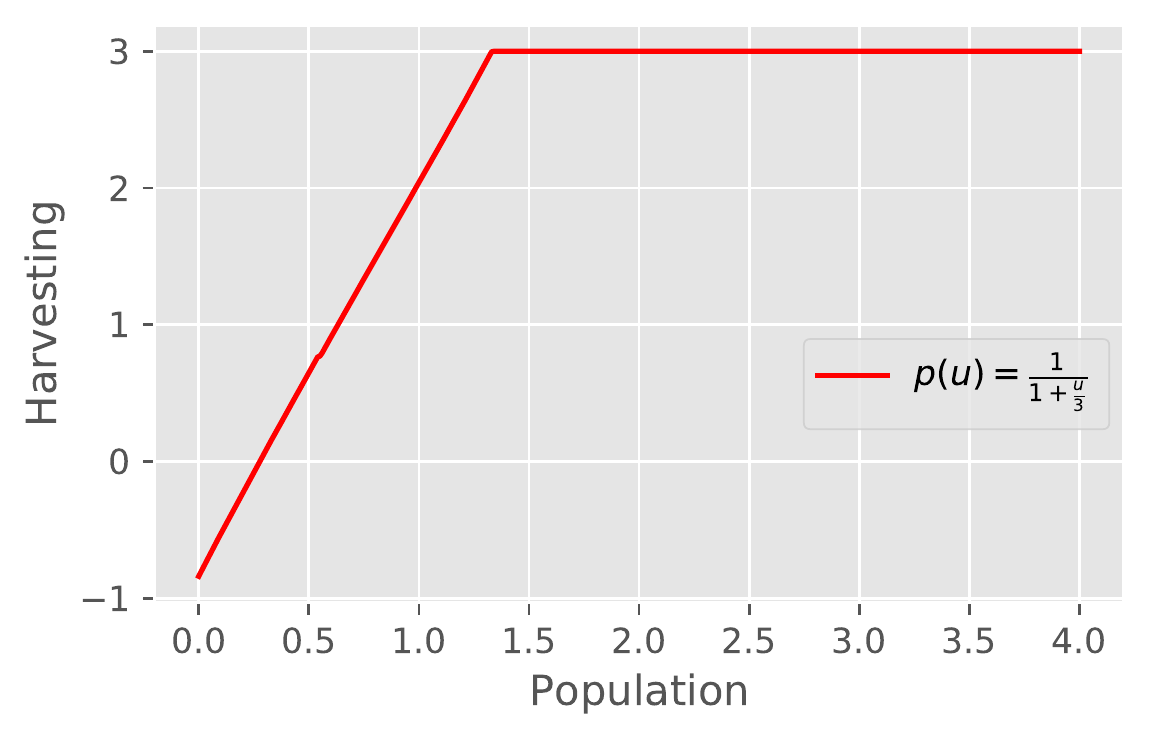}
		}}
		\caption{The left and right panels show the optimal harvesting functions for the model in Section \ref{sec:cost_num}, for two price functions, $p(u) = 1 - u/4$ on the left, and $p(u) = (1+u/3)^{-1}$ on the right. There is no dependency on $\alpha$ here, and $\mu = 2.5$. The other parameters are as before.} \label{fig:price_var_ctrl}
	\end{center}
\end{figure}

\subsection{Interlude with intuition and conjecture}\label{sec:interlude}
\

\noindent
In the previous examples, the interpretation of the control value was ``population units extracted per unit of time."
Based on marginal cost-benefit analysis and intuition, we suspect that a consequence of such a cost formulation, together with others that ensure that the value functions is increasing and concave, is that the optimal harvesting strategies are monotonically increasing in the population level (see also Figures \ref{fig:switching_1}, \ref{fig:switching_cost_1} and \ref{fig:price_var_ctrl}).

Our intuition is as follows: Let's take for granted that the value function is increasing and concave in the population variable.
Consider the cost-benefit of seeding an extra unit of population, or harvesting one less.
If the population increases, the sale value of one future unit of population is the same, but future growth is less favourable at higher population levels, because of intra-species competition.
Taking into account both, we can say that we have a decreasing marginal benefit of having an extra unit of population, as population increases.
We can also see it in the concave shape of the value function.
Its rate of increase per unit of population decreases as population increases.
Moreover, with constant effort harvesting strategies---that is, cost dependent on the absolute value of stocking or harvesting---marginal cost per unit does not depend on the population level.
The difference between marginal benefit and marginal cost is also monotonic decreasing, which means the \textit{net marginal benefit of one unit of population is decreasing in the population}.
This implies that marginally increasing the population level would make stocking less favourable or harvesting more favourable, so we have a monotonic increasing harvesting strategy.

We have not found a way to restrict our general model to provide a clear conjecture statement, mainly because our model only allows for finite rates of harvesting and stocking---which we consider more realistic.
However, there are many suggestive results in the literature that are supportive of our thinking, if we relax this assumption and allow possibly infinite stocking-harvesting rates.
\citet{song2011optimal} show in Theorem 4.4 that the value function is continuous for a model similar to our baseline model, with switching.
\cite{AS98, alvarez2000singular}, working with a Verhulst-Pearl model and a general diffusion model respectively, with no switching, constant price-cost $p$, and a possibly infinite rate of harvesting, show that the value function is continuous and increasing.
\cite{Ky18}, working with a multi-dimensional population system with harvesting and stocking, no switching, state-dependent price-cost $p$, and a possibly infinite rate of harvesting, show in Proposition 2.5 that the value function is increasing and Lipschitz continuous.
It can be easily shown that Proposition 2.5 from \cite{Ky18} can be generalised to a model with switching.


\subsection{Variable effort harvesting examples}\label{sec:var_eff}
\

\noindent
Monotonic strategies are evidently simpler to interpret, to implement, and to parametrize.
But we want to know what happens to a model if control is borne differently.
What if the cost depends on the intensity of extraction, rather than on the extracted quantity itself?
Numerical experiments, e.g., the one in Figure \ref{fig:intensity_var_ctrl}, show that the optimal harvesting strategy may not be monotonic in such cases; for example when the cost of control applies to the fractional harvesting or seeding time rate.
That is because the cost incurred through harvesting or seeding \textit{per unit of population}, per unit of time, can now decrease as the population increases---even with an increasing cost function!
With a possibly decreasing marginal cost of extraction, we may have a non-monotonic net marginal benefit.

This makes the v-shape in the right panel possible.
If the population stock is very low, while growth is favourable, the cost per unit seeded is not favourable at all.
As the population grows, seeding costs per unit go down, so the optimal strategy reflects that.
After further increases, costs stay low, but growth prospects diminish as well. Hence optimal control slowly shifts from seeding to extraction.
The adaptation of our general model for such a type of control was described in Section \ref{sec:var_eff}.
The discussion of the numerical method is detailed in Appendix \ref{sec:appendix_var_effort}.

\begin{figure}[h!tb]
	\begin{center}
		\subfloat{{
			\includegraphics[scale=0.6]{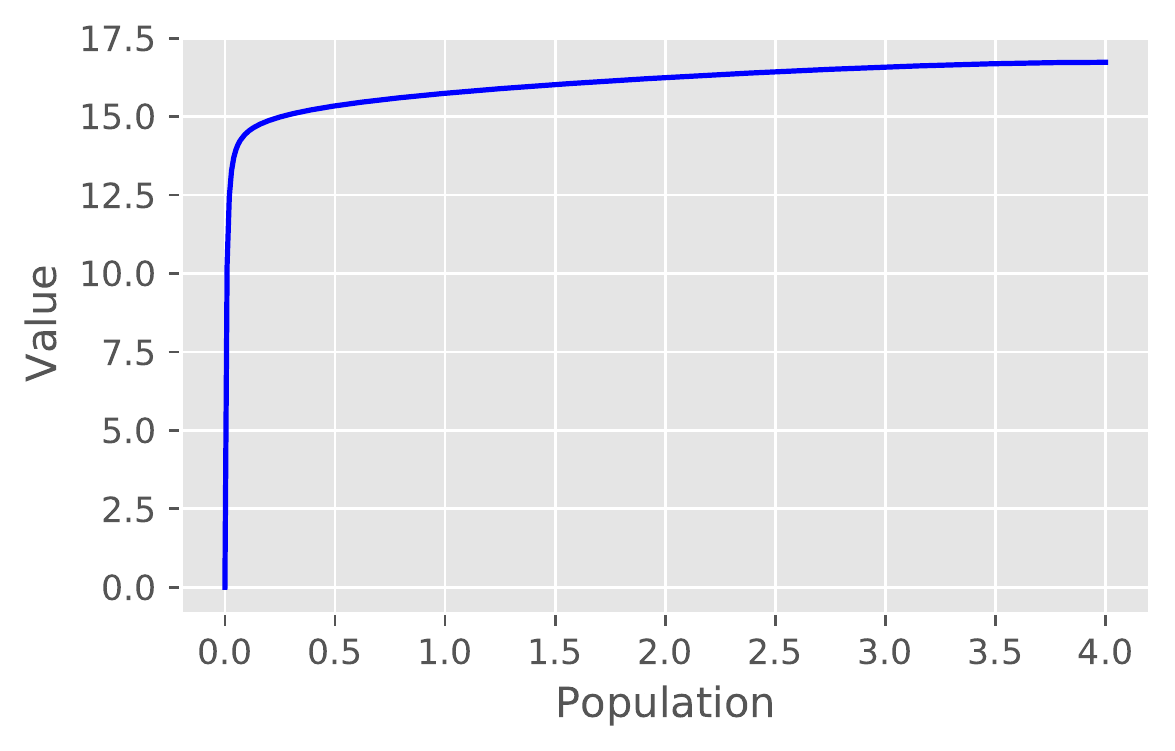}
		}}
		\subfloat{{
			\includegraphics[scale=0.6]{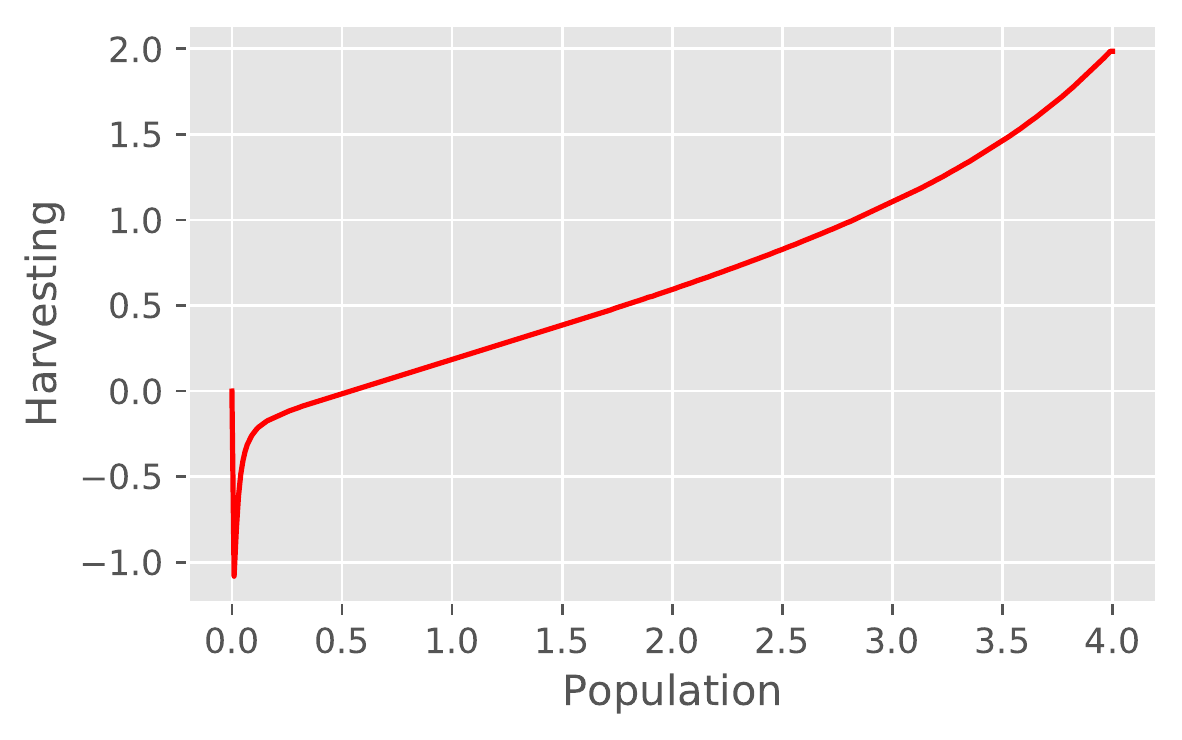}
		}}
		\caption{The left and right panels show the value function and optimal harvesting rate for a model with variable effort harvesting, and switching that affects the population growth rate. $\mu(\alpha) = 2 - \alpha/2$, $\kappa(\alpha) = 1/2$, and the cost function is $u^2$. The other parameters are as in Section \ref{sec:intro_lv}. For simplicity, we show only the value function and the harvesting strategy for $\alpha=1$. The other state has similar values and controls.} \label{fig:intensity_var_ctrl}
	\end{center}
\end{figure}

\subsection{Stochastic pricing examples}\label{sec:stoch_ex}
\

\noindent
A realistic model of population management may want to take into account that the value of the resource varies over time.
We have already considered price dependence on gross quantity extracted.
But price can also change over time because of external factors.
For example, usually economists model the price of a commodity as a stochastic variable \citep{osborne1959brownian, miltersen2003commodity}.
If this variability is large enough, we would like to understand what effects it has on our optimal population management strategy.
The following is not the standard model of price stochasticity used in finance, but it is convenient for our set-up.

Consider the price-cost function
\begin{equation*}
	p(t, x, \alpha, u) = \left( p_0(\alpha, u) + \Phi(t)\right)u - C(t,x,\alpha,u), \quad (x, \alpha)\in \mathbb{R}_+ \times \mathcal{M}, t\in \mathbb{R}^+,
\end{equation*}
where $p_0(\cdot) = 1$,  and $C(t, x, \alpha, u) = u^2/2$.
If we ignore random perturbations, the evolution of $\Phi\cd$ is given by
\begin{equation*}
	d\Phi(t)=\Phi(t)\big(0.4-\Phi(t)\big)dt, \quad \Phi(0)\in [0, 0.4).
\end{equation*}
It can be seen that $\Phi(t)\in (0, 0.4)$ for any $t\ge 0$ and $\Phi(t)\to 0.4$ as $t\to \infty$.
Now we suppose that the evolution of $\Phi(t)$ is subject to a white noise $w_0\cd$ and is given by
\begin{equation*}
	d\Phi(t) = \Phi(t)\big(0.4-\Phi(t)\big)dt + 0.5\Phi(t)\big(0.4-\Phi(t)\big)dw_0(t).
\end{equation*}
We also have $\Phi(t)\in [0, 0.4)$ with probability one.

\begin{figure}[h!tb]
	\begin{center}		
		\subfloat{{\includegraphics[scale=0.6]{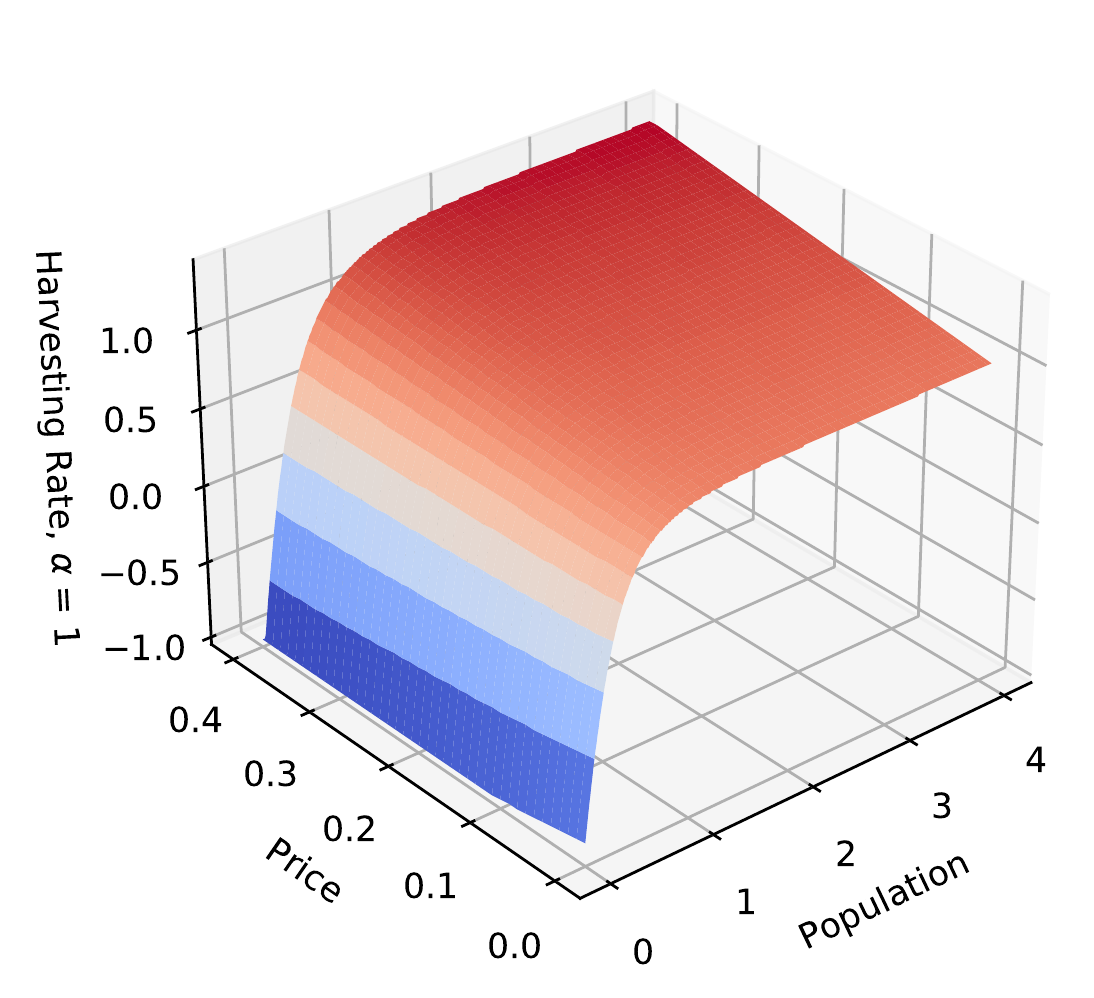} }}
		\subfloat{{\includegraphics[scale=0.5]{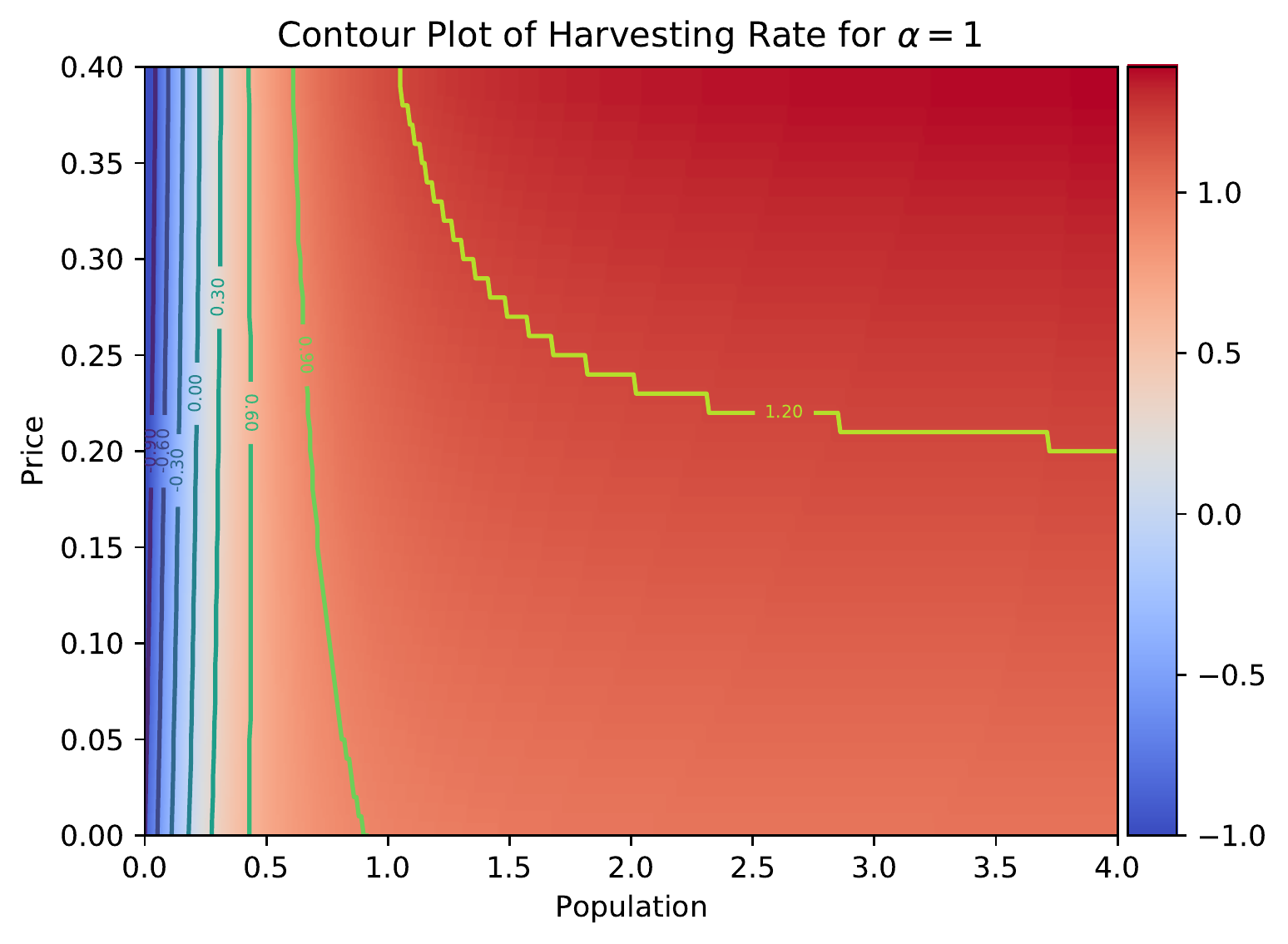} }}%
		\caption{Harvesting rate in regime $\alpha=1$, 3D plot on the left, contour plot on the right. Stochastic price given by $p_0(\alpha, u) + \Phi(t)$, cost by $C(u) = u^2/2$, and the other parameters as in the baseline model. Discussion in Section \ref{sec:stoch_ex}.}
		\label{fig:stoch_p}
	\end{center}
\end{figure}	

The harvesting policies for the two regimes share the same shape, so Figure \ref{fig:stoch_p} shows the harvesting rate as a function of population size $x$ and the observation of the price $\Phi(\cdot)$, for one of the regimes, $\alpha=1$, for brevity.
The value function has the usual properties, and we omit its graph.
It is increasing and concave in population for all price cross-sections in the range.
Moreover, fixing the population variable, we see that higher prices are more conducive to both harvesting and seeding, as expected.

\subsection{Seasonal and periodic variability}\label{sec:period_ex}
\

\noindent
Finally, in this subsection, we consider applications of our method to models with periodic parameters.
Here, periodicity can be thought of as representing seasonal changes in the environment.
In this example, we consider a simple adaptation of the Verhulst model, where growth is additively affected by a sine wave.
There is little justification for this functional form, except to point out that our sine wave would be one of the two likely lowest order terms in the Fourier expansion of any periodic dependency we may conceive.

Suppose $b(t, x, \alpha) = x\left( 4 - \alpha + \sin (2\pi t) -2x \right)$, $\sigma(t, x, \alpha) = x$, $P(t, x, \alpha, u) = 1$, and $ C(t, x, \alpha, u) = u^2/2 $, where $(t, x, \alpha, u)\in \mathbb{R}^2_+\times \{1, 2\}\times \mathcal{U}$.
To help with numerical estimation, we reduce the population range, the control range and the sampling density.
Now $\mathcal{U}=\{u :u = k/20, k \in \mathbb{Z}, -20 \le k \le 40\}$.
The time step $h_1$ is $T/4000$.
The other parameters are as before.
Note that $b(\cdot, x, \alpha)$, $\sigma(\cdot, x, \alpha)$,  and $p(\cdot, x, \alpha, u)$ are periodic with period $T=1$.

\begin{figure}[h!tb]
	\centering		
	\subfloat{{\includegraphics[scale=0.6]{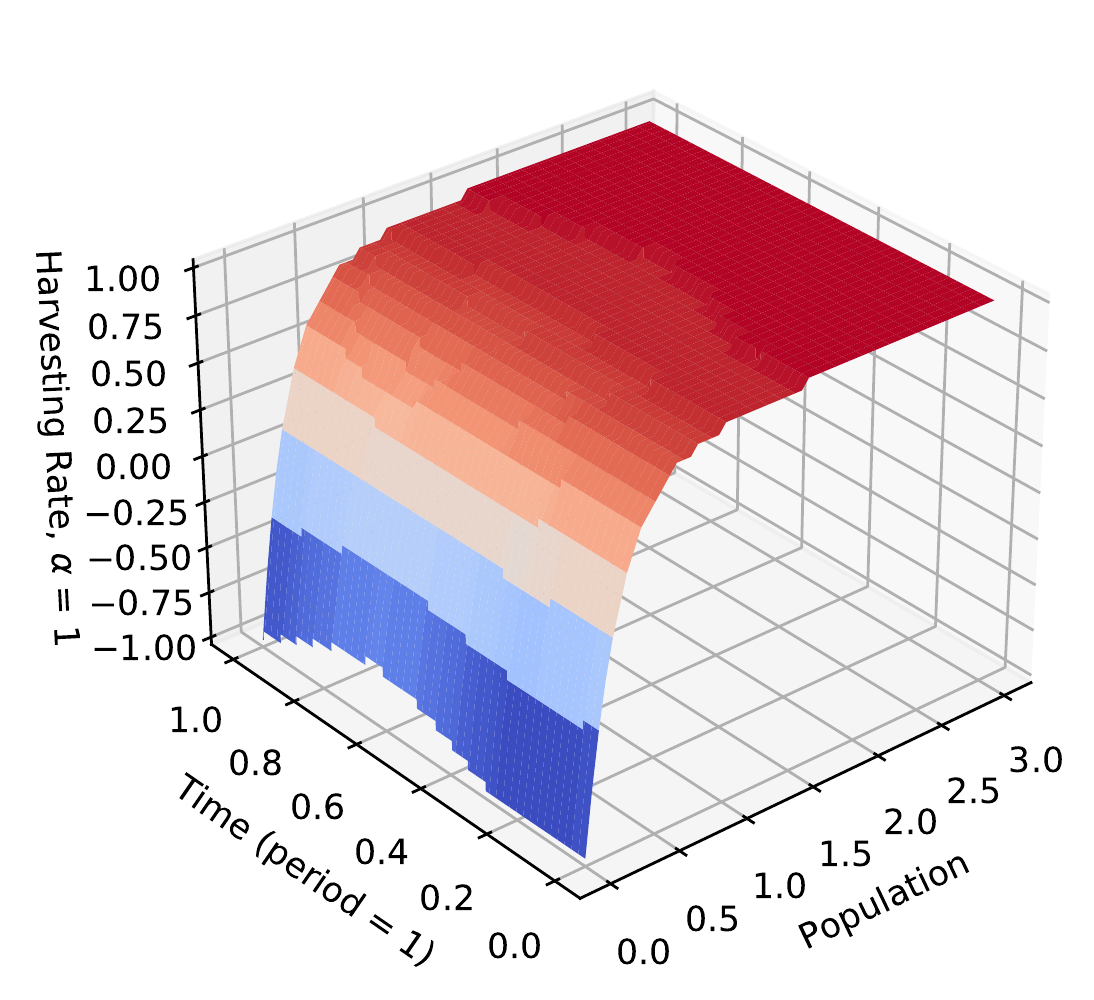} }}%
	\subfloat{{\includegraphics[scale=0.5]{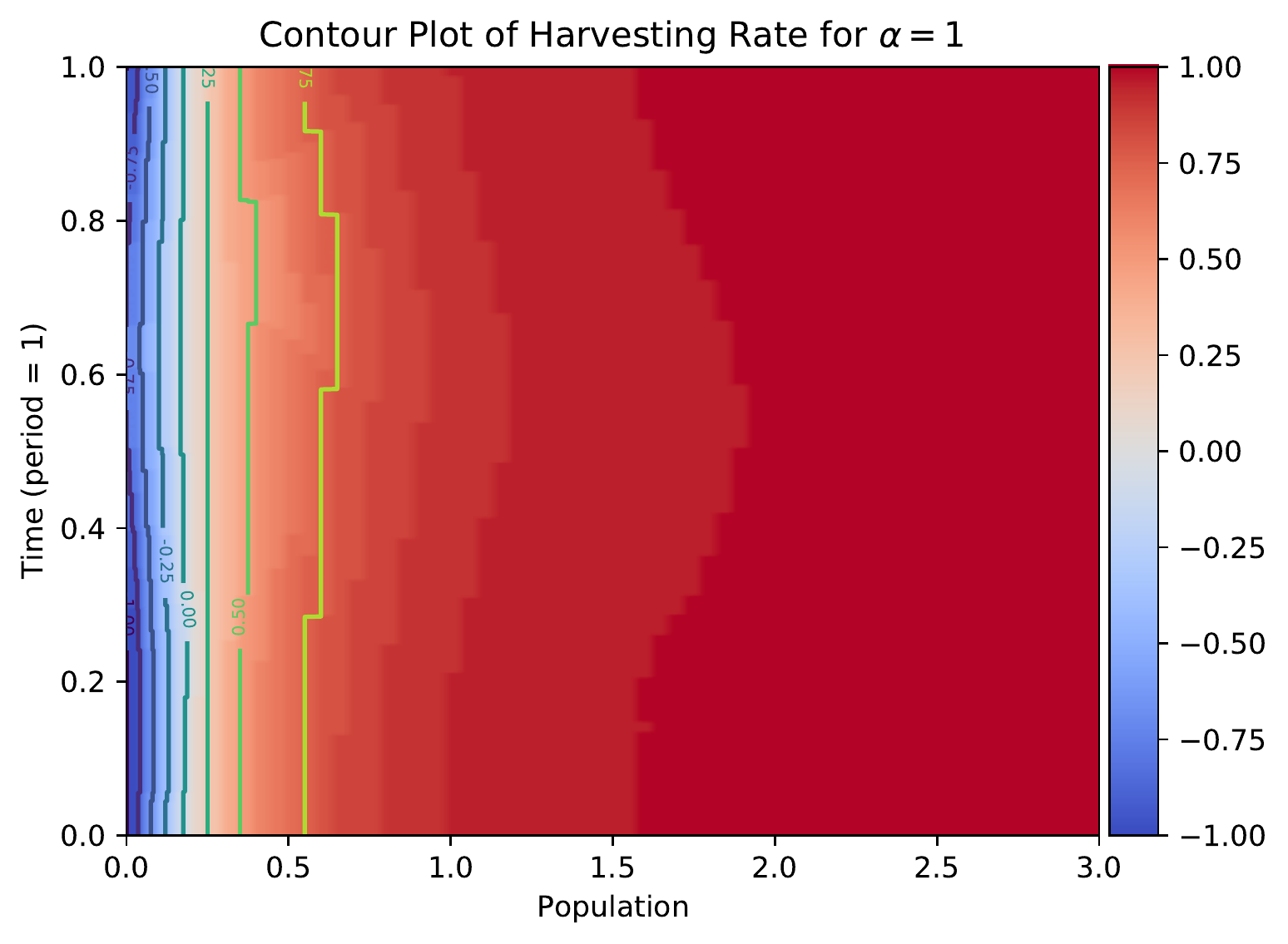} }}%
	\caption{Harvesting rate in regime $\alpha=1$, 3D plot on the left, contour plot on the right. This corresponds to an adapted Verhulst model, with drift affected by a sine wave, $b(t, x, a) = x (4 - \alpha - \sin(2 \pi t) - 2x)$, $\sigma(t,x,\alpha) = x$, constant pricing and quadratic cost, $C(u) = u^2/2$. Discussion in subsection \ref{sec:period_ex}.}
	\label{fig:periodicity}
\end{figure}	

Figure \ref{fig:periodicity} shows the estimation results for the harvesting rate as a function of population size $x$ and the time $t$.
As we can see the periodicity of the growth rate is apparent in the graphs.
It is interesting that the maximal and minimal effects on the harvesting rate do not match the maximum and minimum of the sine wave at $t=1/4$ and $t=3/4$.
Not only is the effect of harvesting out of phase with the sine wave, the effects on harvesting and stocking seem to be out of phase with each other.
At this point, we are not sure if this observation is an artefact of the limited sampling in the numerical estimation, or a robust result.
We omit the value function and the harvesting regime for $\alpha = 2$ for brevity.

\begin{figure}[h!tb]
	\centering		
	\subfloat{{\includegraphics[scale=0.6]{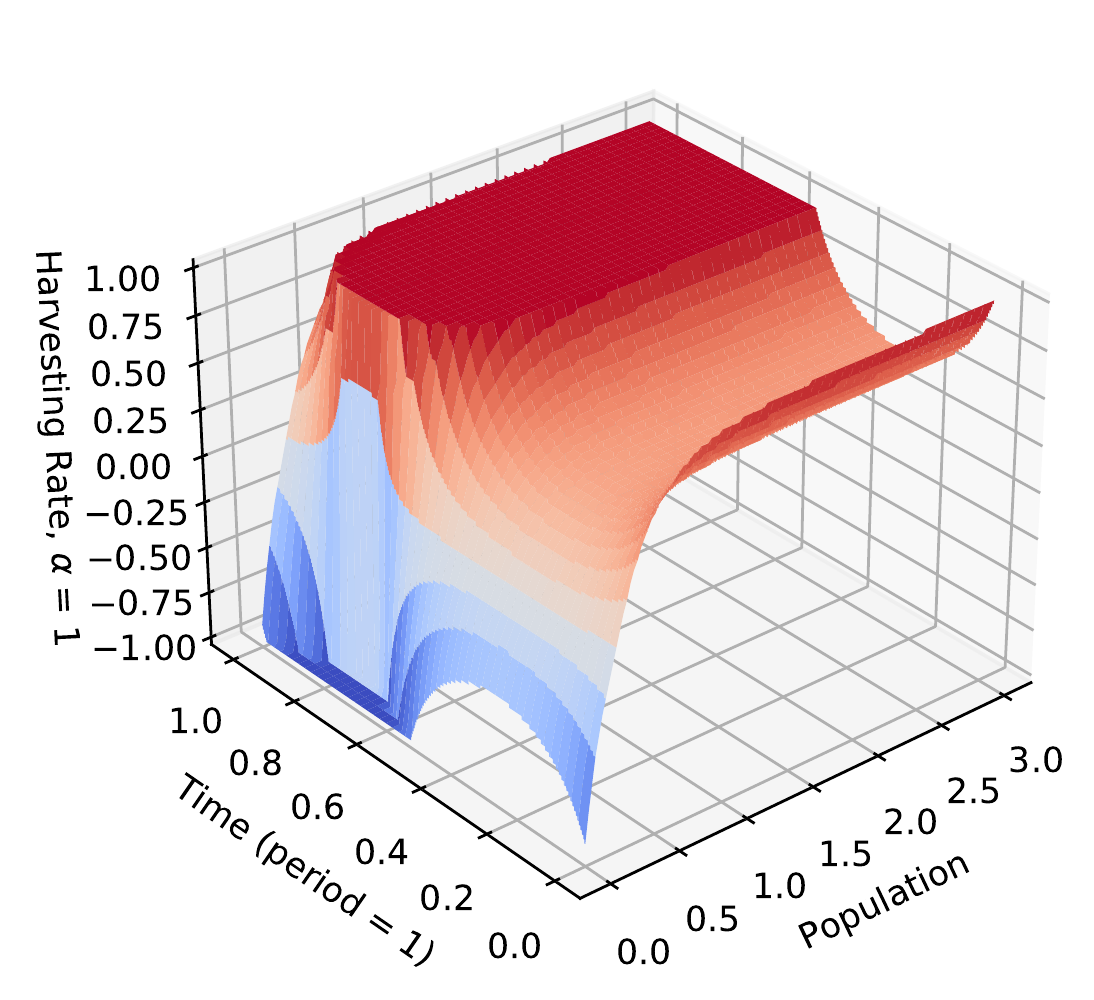} }}%
	\subfloat{{\includegraphics[scale=0.5]{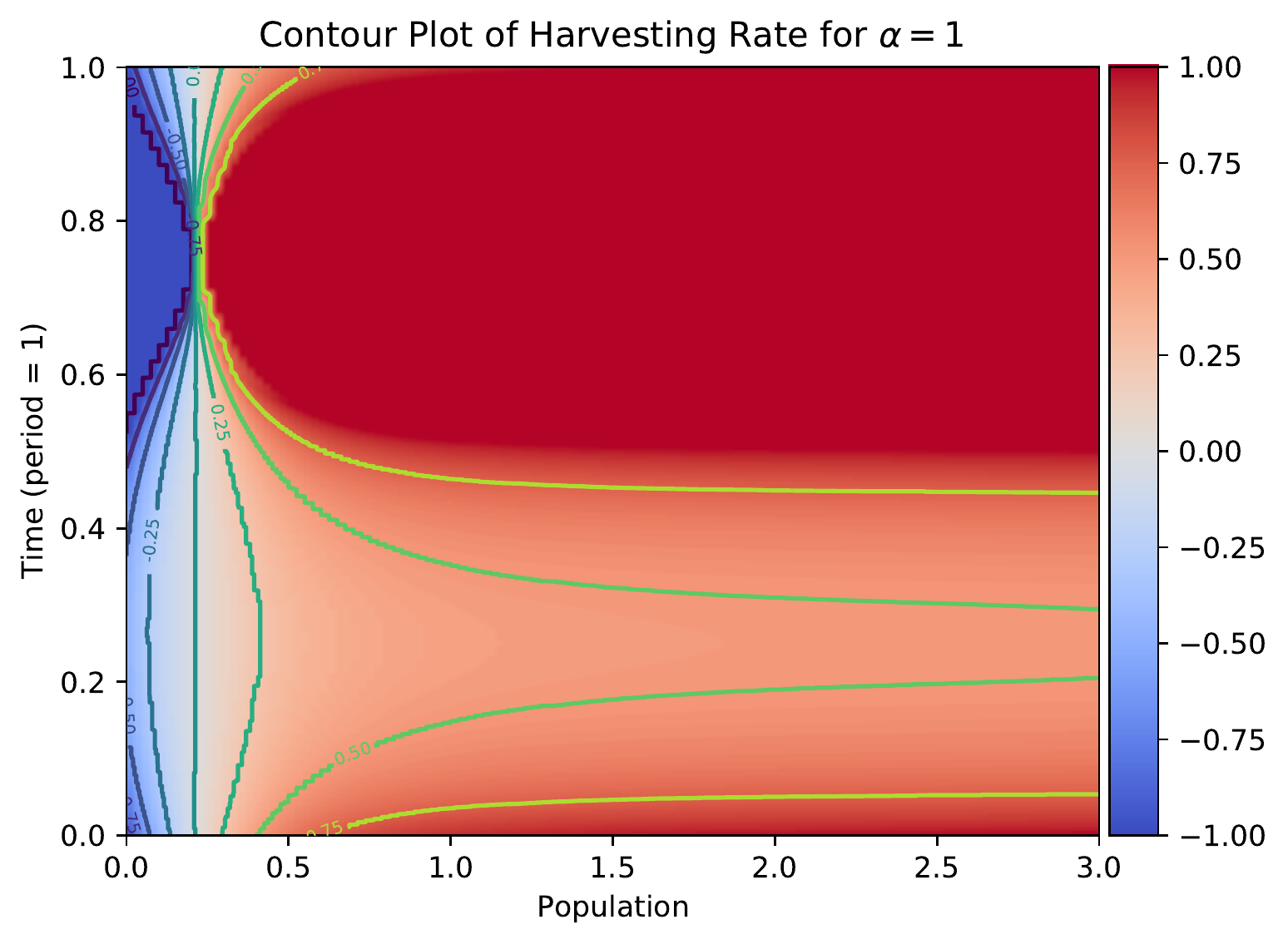} }}%
	\caption{Harvesting rate in regime $\alpha=1$, 3D plot on the left, contour plot on the right. The parameters correspond to the baseline Verhulst model, $b(t, x, a) = x (4 - \alpha - 2x)$, $\sigma(t,x,\alpha) = x$, price is constant, and the quadratic cost is affected by a sine wave, $C(t,u) = (1 + \sin(2 \pi t)) u^2/2$. $h_1 = T/16000$, and $\mathcal U \in [-1,-1]$. Discussion in subsection \ref{sec:period_ex}.}
	\label{fig:cost_periodicity}
\end{figure}	

Seasonality can affect other parameters as well.
In the following example, we have seasonality affecting costs, but not the growth of the population.
Figure \ref{fig:cost_periodicity} shows what are intuitive results.
Harvesting when the cost is high is low, and vice-versa.
Again, we omit the value function and the harvesting regime for $\alpha = 2$ for brevity.

\section{Discussion}\label{sec:disc}

The objective of this paper was to develop a theoretically sound grounding for numerical methods, that would allow us to analyze realistic, complex, stochastic population models.
Furthermore, we wanted to use these methods to explore the effects of adding features or extensions to classic models like Verhulst-Pearl, Gompertz and Nisbet.
Among the extensions of interest that we explored are Markovian environmental switching, periodicity, and non-trivial harvesting and stocking cost and price dependencies.
These are very important in any realistic model of population management.

We showed that the optimal harvesting and stocking strategy is not always of bang-bang type---that is, all or nothing.
The crucial factor seems to be whether the dependence of the cost function is convex in the harvested time rate or not.
A similar observation applies to price dependency.
Regarding environmental switching, we determined that, in the limit of switching rates favouring one state versus another, the value function and harvesting rate converge as expected to the ones of the favoured state.
Moreover, in the limit of fast switching, the value function and harvesting rate converge to the one representing an ``average environment," as is expected from the slow-fast dynamics analysis of stochastic systems without harvesting \citep{YZ09, HL20, DHNY21}.

In general, we have monotonicity of the value function in the population size, but also of the harvesting rate.
The harvesting rate is increasing in the population size, which is intuitively expected, as a higher population should be more suitable for harvesting, everything else being equal.
However, this observation breaks down when the cost of harvesting depends on the proportion harvested from the available population per unit of time, as in the so-called variable effort strategies.

Seasonal variation in the model parameters also has implications that are expected or intuitive.
The maximum harvesting or stocking values correspond to times of maximal fertility.\footnote{Observe that we have not considered models with a distinction between young and old.
All individuals are the same at all points in time, and have the same harvested value.}
In the slow interval of the season, both harvesting and stocking are reduced.

Allowing the harvesting strategy to depend on a stochastic price variable leads to a variable strategy where harvesting is favoured by high prices, and stocking is favoured by low prices, as expected.

Overall, we think that proper numerical techniques are essential to test model validity by numerical experiments, especially when the complexity of the model does not favour finding explicit solutions.
Moreover, such numerical methods are good at finding points of interest in parameter space, and for formulating conjectures.
\\

{\bf Acknowledgments.}  A. Hening is supported by the NSF through the grant DMS 1853463. K. Q. Tran is supported by the National Research Foundation of Korea grant funded by the Korea Government (MIST) NRF-2021R1F1A1062361.

\clearpage
\bibliographystyle{agsm}
\bibliography{harvest}


\appendix

\section{Transition Probabilities} \label{sec:ap_a}

\subsection{The formulation from Section \ref{sec:for}}
\

\noindent
We first look at the details we need for the setting from Section \ref{sec:for}.
With the notation defined in Section \ref{sec:num},  let $(x, \al, u)\in S_h\times \M\times \mathcal{U}$ and
denote by $\E^{h, u}_{x, \al, n}$, $\Cov^{h, u}_{x, \al, n}$ the conditional expectation and covariance given by
$$\{X_m^h, \al^h_m, U_m^h, m\le n, X_n^h=x, \al^h_n=\al, U^h_n=u\},$$
respectively.  Define $\Delta X^h_n = X^h_{n+1}-X^h_n$.
Our objective in this subsection is to define transition probabilities $q^h ((x,k), (y, l) | u)$ so that the controlled Markov chain $\{(X^h_n, \al^h_n)\}$ is locally consistent with respect to the controlled diffusion \eqref{e.2}. By this we mean that the following conditions hold:
\beq{a.1}
\barray
\aad \E^{h, u}_{x, k, n}\Delta X_n^h = ({b}(x, k) - u) \Delta t^h(x, k, u) + o(\Delta t^h(x, k, u)),\\
\aad \Var^{h, u}_{x, k, n}\Delta X_n^h = \sg^2(x, k, u)\Delta t^h(x, k, u) + o(\Delta t^h(x, k, u)),\\
\aad \P^{h, u}_{x, k, n}(\al^h_{n+1}=l)= q_{kl}\Delta t^h(x, k, u)+o(\Delta t^h(x, k, u)) \quad \text{for} \quad l\ne k,\\
\aad \P^{h, u}_{x, k, n}(\al^h_{n+1}=k)=1+ q_{kk}\Delta t^h(x, k, u)+o(\Delta t^h(x, k, u)),\\
\aad \sup\limits_{n, \ \omega} |\Delta X_n^h| \to 0 \quad \text{as}\quad h \to 0.
\earray
\eeq
Using the procedure developed by \cite{Kushner90},
for $(x, \al)\in S_h\times \M$ and $u\in \mathcal{U}$,  we define
\beq{a.2}
\barray
\aad Q_h (x, k, u)= \sg^2(x, k)+h |b(x, k)-u| -h^2q_{kk}+h,\\
\aad q^h \((x, k), (x+h, k) |u\) =\dfrac{
\sg^2(x, k)/2+\big(b(x, k)-u\big)^+ h }{Q_h (x, k, u)}, \\
\aad q^h \((x, k), (x-h, k) | u\) =\dfrac{
\sg^2(x, k)/2+\(b(x, k)-u\)^- h}{Q_h (x, k, u)}, \\
\aad   q^h \( (x, k), (x, l) | u\) =\dfrac{h^2 q_{kl} }{ Q_h (x, k, u)} \quad \text{ for }\quad k\ne l, \\
\aad q^h \( (x, k), (x, k) | u\) =\dfrac{h  }{ Q_h (x, k, u)},\quad \Delta t^h (x, k, u)=\dfrac{h^2}{Q_h(x, k, u)},
\earray
\eeq
where for a real number $r$,  $r^+=\max\{r, 0\}$,
$r^-=-\min\{0, r\}$.
Set $q^h \((x, k), (y, l)|u\)=0$ for all unlisted values of $(y, l)\in S_h\times \M$.
Note that  $ \sup_{x, k, u} \Delta t^h (x, k, u)\to 0$ as $h\to 0$.
 Using the above transition probabilities,
 we can check that the locally consistent conditions of $\{(X^h_n, \al^h_n)\}$ are satisfied.

\begin{lem}
	The Markov chain $\{(X^h_n, \al^h_n)\}$ with transition probabilities $\{q^h\cd\}$ defined in \eqref{a.2} satisfies the local consistence in \eqref{a.1}.
\end{lem}

\subsection{Variable effort harvesting-stocking strategies}\label{sec:appendix_var_effort}
\

\noindent
For $(x, \al, u)\in S_h\times \M\times \mathcal{U}$, let $\E^{h, u}_{x, \al, n}$, $\Cov^{h, u}_{x, \al, n}$ denote the conditional expectation and covariance given by
$$\{X_m^h, \al^h_m, U_m^{ h}, m\le n, X_n^h=x, \al^h_n=\al, U^{h}_n=u\},$$
respectively. Define $\Delta X^h_n = X^h_{n+1}-X^h_n$. In order to approximate the process $(X\cd, \al\cd)$ given in \eqref{f.1},
 the controlled Markov chain $\{(X^h_n, \al^h_n)\}$ must be locally consistent with respect to $(X\cd, \al\cd)$ in the sense that the following conditions hold
\beq{aa.1}
\barray
\aad \E^{h, u}_{x, k, n}\Delta X_n^h = ({b}(x, k) - ux) \Delta t^h(x, k, u) + o(\Delta t^h(x, k, u)),\\
\aad \Var^{h, u}_{x, k, n}\Delta X_n^h = \sg^2(x, k, u)\Delta t^h(x, k, u) + o(\Delta t^h(x, k, u)),\\
\aad \P^{h, u}_{x, k, n}(\al^h_{n+1}=l)= q_{kl}\Delta t^h(x, k, u)+o(\Delta t^h(x, k, u)) \quad \text{for} \, l\ne k,\\
\aad \P^{h, u}_{x, k, n}(\al^h_{n+1}=k)=1+ q_{kk}\Delta t^h(x, k, u)+o(\Delta t^h(x, k, u)),\\
\aad \sup\limits_{n, \ \omega} |\Delta X_n^h| \to 0 \quad \text{as}\quad h \to 0.
\earray
\eeq
To this end, we define the
 transition probabilities $q^h ((x,k), (y, l) | u)$ as follows. For
$(x, k)\in S_h\times \M$ and $u\in \mathcal{U}$, let
\beq{aa.2}
\barray
\aad Q_h (x, k, u)= \sg^2(x, k)+h |b(x, k)-ux| -h^2q_{kk}+h,\\
\aad q^h \((x, k), (x+h, k) |u\) =\dfrac{
\sg^2(x, k)/2+\big(b(x, k)-ux)^+ h }{Q_h (x, k, u)}, \\
\aad q^h \((x, k), (x-h, k) | u\) =\dfrac{
\sg^2(x, k)/2+\(b(x, k)-ux\)^- h}{Q_h (x, k, u)}, \\
\aad   q^h \( (x, k), (x, l) | u\) =\dfrac{h^2 q_{kl} }{ Q_h (x, k, u)} \quad \text{ for }\quad k\ne l, \\
\aad q^h \( (x, k), (x, k) | u\) =\dfrac{h  }{ Q_h (x, k, u)},\quad \Delta t^h (x, k, u)=\dfrac{h^2}{Q_h(x, k, u)}.
\earray
\eeq
Set $q^h \((x, k), (y, l)|u\)=0$ for all unlisted values of $(y, l)\in S_h\times \M$.

\subsection{Uncertain price functions}
\

\noindent
Let $(\phi, x, \al, u)\in \wdh S_h\times \M\times \mathcal{U}$ and denote by $\E^{h, u}_{\phi, x, \al, n}$, $\Cov^{h, u}_{\phi, x, \al, n}$ the conditional expectation and covariance given by
$$\{\Phi^h_m, X_m^h, \al^h_m, U_m^{h}, m\le n, \Phi^h_n=\phi, X_n^h=x, \al^h_n=\al, U^{h}_n=u\},$$
respectively. Define $\Delta X^h_n = X^h_{n+1}-X^h_n$ and $\Delta \Phi^h_n = \Phi^h_{n+1}-\Phi^h_n$.
In order to approximate $(\Phi\cd, X\cd, \al\cd)$ given by \eqref{e.2}-and-\eqref{g.1},
 the controlled Markov chain $\{(\Phi^h_n, X^h_n, \al^h_n)\}$ must be locally consistent with respect to $(\Phi\cd, X\cd, \al\cd)$
in the sense that the following conditions hold:
\beq{b.1}
\barray
\aad \E^{h, u}_{\phi, x, k, n}\Delta X_n^h = ({b}(x, k) - u) \Delta t^h(\phi, x, k, u) + o(\Delta t^h(\phi, x, k, u)),\\
\aad \Var^{h, u}_{\phi, x, k, n}\Delta X_n^h = \sg^2(x, k, u)\Delta t^h(\phi, x, k, u) + o(\Delta t^h(\phi, x, k, u)),\\
\aad \E^{h, u}_{\phi, x, k, n}\Delta \Phi_n^h = {b}_0(x, k) \Delta t^h(\phi, x, k, u) + o(\Delta t^h(\phi, x, k, u)),\\
\aad \Var^{h, u}_{\phi, x, k, n}\Delta \Phi_n^h = \sg^2_0(x, k)\Delta t^h(\phi, x, k, u) + o(\Delta t^h(\phi, x, k, u)),\\

\aad \P^{h, u}_{\phi, x, k, n}(\al^h_{n+1}=l)= q_{kl}\Delta t^h(\phi, x, k, u)+o(\Delta t^h(\phi, x, k, u)) \quad \text{for} \quad l\ne k,\\
\aad \P^{h, u}_{\phi, x, k, n}(\al^h_{n+1}=k)=1+ q_{kk}\Delta t^h(\phi, x, k, u)+o(\Delta t^h(\phi, x, k, u)),\\
\aad \sup\limits_{n, \ \omega} \big( |\Delta X_n^h|  + |\Delta \Phi^h_n|\big)\to 0 \quad \text{as}\quad h \to 0.
\earray
\eeq
To this end, we define the  transition probabilities $q^h ((\phi, x,k), (\psi, y, l) | u)$ as follows. For
$(\phi, x, k)\in \wdh S_h\times \M$ and $u\in \mathcal{U}$,  let
\beq{b.2}
\barray
\aad Q_h (\phi, x, k, u)=\sg^2(x, k) +h |b(x, k)-u| +\sg_0^2 (x, k)  +h |b_0 (x, k)|-h^2q_{kk}+h,\\
\aad q^h \((\phi, x, k), (\phi, x+h, k) |u\) =
\dfrac{\sg^2(x, k)/2+\big(b(x, k)-u\big)^+ h }{Q_h (\phi, x, k, u)}, \\
\aad q^h \((\phi, x, k), (\phi, x-h,  k) | u\) =
\dfrac{\sg^2(x, k)/2+\(b(x, k)-u\)^- h}{Q_h (\phi, x, k, u)}, \\
\aad   q^h \( (\phi, x, k), (\phi, x, l) | u\) =\dfrac{h^2 q_{kl} }{ Q_h (\phi, x, k, u)} \quad \text{for} \quad k\ne l,\\
\aad q^h \( (\phi, x, k), (\phi, x, k) | u\) =\dfrac{h  }{ Q_h (\phi, x, k, u)}, \quad \Delta t^h (\phi, x, k, u)=\dfrac{h^2}{Q_h(\phi, x, k, u)},\\
\aad q^h \((\phi, x, k), (\phi+h, x, k) |u\) =
\dfrac{ \sg^2_0 (x, k)/2  +h  b^+_0 (x, k)}{Q_h (\phi, x, k, u)},\\
\aad q^h \((\phi, x, k), (\phi-h, x, k) |u\) =
\dfrac{\sg_0^2 (x, k)/2  +h  b^-_0 (x, k) }{Q_h (\phi, x, k, u)}.
\earray
\eeq
Set $q^h \((\phi, x, k), (\psi, y, l)|u\)=0$ for all unlisted values of $(\psi, y, l)\in \wdh S_h\times \M$.

\subsection{The combined effects of seasonality and Markovian switching}
\

\noindent
Recall that $\wdt S_{h}: = \{(\ga, x)=(k_1 h_1, k_2 h_2)'\in \rr^2: k_i\in \mathbb{Z}_{\ge 0}, k_1\le T/h_1\}.$
 Let $(\ga, x, \al, u)\in \wdt S_h\times \M\times \mathcal{U}$ and denote by $\E^{h, u}_{\ga, x, \al, n}$, $\Cov^{h, u}_{\ga, x, \al, n}$ the conditional expectation and covariance given by
$$\{\Ga^h_m, X_m^h, \al^h_m, U_m^{h}, m\le n, \Ga^h_n=\ga, X_n^h=x, \al^h_n=\al, U^{h}_n=u\},$$
respectively. Define $\Delta X^h_n = X^h_{n+1}-X^h_n$ and $\Delta \Ga^h_n = \Ga^h_{n+1}-\Ga^h_n$.
In order to approximate $(\Ga\cd, X\cd, \al\cd)$ given by \eqref{h.1},  the controlled Markov chain $\{(\Ga^h_n, X^h_n, \al^h_n)\}$ must be locally consistent with respect to $(\Ga\cd, X\cd, \al\cd)$
in the sense that the following conditions hold:
\beq{bb.1}
\barray
\aad \E^{h, u}_{\ga, x, k, n}\Delta X_n^h = ({b}(\ga, x, k) - u) \Delta t^h(\ga, x, k, u) + o(\Delta t^h(\ga, x, k, u)),\\
\aad \Var^{h, u}_{\ga, x, k, n}\Delta X_n^h = \sg^2(\ga, x, k, u)\Delta t^h(\ga, x, k, u) + o(\Delta t^h(\ga, x, k, u)),\\
\aad \Delta \Ga_n^h = \Delta t^h(\ga, x, k, u),\quad \text{if} \quad \ga + \Delta t^h(\phi, x, k, u)<T,\\
\aad \Delta \Ga_n^h = \ga + \Delta t^h(\ga, x, k, u)-T \quad \text{if} \quad \ga + \Delta t^h(\ga, x, k, u)\ge T,\\
\aad \P^{h, u}_{\ga, x, k, n}(\al^h_{n+1}=l)= q_{kl}\Delta t^h(\ga, x, k, u)+o(\Delta t^h(\ga, x, k, u)) \quad \text{for } \, l\ne k,\\
\aad \P^{h, u}_{\ga, x, k, n}(\al^h_{n+1}=k)=1+ q_{kk}\Delta t^h(\ga, x, k, u)+o(\Delta t^h(\ga, x, k, u)),\\
\aad \sup\limits_{n, \ \omega} \big( |\Delta X_n^h| + \Delta \Ga^h_n \big) \to 0 \quad \text{as}\quad h \to 0.
\earray
\eeq
To this end, we define the
 transition probabilities $q^h ((\ga, x,k), (\lambda, y, l) | u)$ as follows. Let
$(\ga, x, k)\in \wdt S_h\times \M$ and $u\in \mathcal{U}$. If $\ga+h_1=T$, $\ga+h_1$ in the following definition is understood as $0$. Let
\beq{bb.2}
\barray
\ad \Delta t^h(\ga, x, k, u)=h_1,\\
\ad q^h \((\ga, x, k), (\ga + h_1, x+ h_2, k) | u\) =
\dfrac{\Big( \sg^2(\ga, x, k)/2+\big(b(\ga, x, k)-u\big)^+ h_2 \Big)h_1}{h_2^2}, \\
\ad q^h \((\ga, x, k), (\ga + h_1, x-h_2, k) | u\) =
\dfrac{\Big( \sg^2(\ga, x, k)/2+\big(b(\ga, x, k)-u\big)^- h_2 \Big)h_1}{h_2^2}, \\
\ad  q^h\big((\ga, x, k), (\ga+h_1, x, l)|u\big)=h_1q_{kl},  \quad \text{for} \quad l\ne k,\\
\ad q^h\big( (\ga, x, k), (\ga+h_1, x, k)|u \big) = 1 -  q^h\big( (\ga, x,k), (\ga+h_1, x+ h_2, k)| u \big) \\
\ad \qquad -  q^h\big( (\ga, x,k), (\ga+h_1, x- h_2, k)| u \big) - \sum\limits_{l\ne k} q^h\big( (\ga, x, k), (\ga+h_1, x, l)|u \big).\earray
\eeq
Set $q^h \((\ga, x, k), (\lambda, y, l)|u\)=0$ for all unlisted values of $(\lambda, y, l)\in \wdt S_h\times \M$.

\section{Continuous--time interpolation}
We will present the convergence analysis for the formulation in Section \ref{sec:for}. The other formulas can be handled in a similar way. Our procedure and methods are similar to those in \cite{Kushner90, Kushner92, Song06}.
The convergence result is based on a continuous-time interpolation of the controlled Markov chain,
which will be  constructed to be piecewise constant on the time interval $[t^h_n, t^h_{n+1}), n\ge 0$.
To this end, we define
$n^h(t)=\max\{n: t^h_n\le t\}, t\ge 0.$
The piecewise constant interpolation of $\{(X^h_n,\al^h_n, U^h_n)\}$, denoted by $\big(X^h(t),\al^h(t), U^h(t)\big)$ is naturally defined as
\beq{d.1}
\barray
\aad X^h(t) = X^h_{n^h(t)}, \quad \al^h(t) = \al^h_{n^h(t)}, \quad U^h(t) = U^h_{n^h(t)}, \quad t\ge 0.
\earray
\eeq
Define $\mathcal{F}^h(t)=\sigma\{X^h(s), \al^h(s), U^h(s): s\le t\}=\mathcal{F}^h_{n^h(t)}$. Also define $$M^h(0)=0, \quad M^h(t)  = \sum\limits_{m=0}^{n^h(t)-1} (\Delta X^h_m -
\E^{h}_m \Delta X_m)  \quad \text{for}\quad t\ge 0.$$
It is obvious that
\beq{d.2}
X^h(t)
= x + \sum\limits_{m=0}^{n^h(t)-1}
\E^{h}_m \Delta X^h_m + M^h(t).
\eeq
Recall that $\Delta t^h_m = h^2/Q_h(X^h_m, \al^h_m, U^h_m)$. It follows that
\beq{d.3}
\barray
\sum\limits_{m=0}^{n^h(t)-1}
\E^{h}_m \Delta X^h_m \ad = \sum\limits_{m=0}^{n^h(t)-1} \[b (X^h_m, \al^h_m)+U^h_m\]\Delta t^h_m\\
\ad=\int_0^t \[b (X^h(s), \al^h(s)) + U^h(s)\] ds-\int_{t^h_{n^h(t)}}^t \[b (X^h(s), \al^h(s))+U^h(s)\] ds\\
\ad = \int_0^t \[b (X^h(s), \al^h(s))+U^h(s)\] ds + \e^h_1(t),
\earray
\eeq
with $\{\e_1^h\cd\}$
 being an $\mathcal{F}^h(t)$-adapted process satisfying \begin{equation*}\lim\limits_{h\to 0} \sup\limits_{t\in [0, T_0]}\E|\e_1^h(t)|=0 \quad \text{for any }0<T_0<\infty.\end{equation*} For simplicity, we suppose that
$\inf\limits_{(x, k)}1/|\sg(x, k)|>0$ (if this is not the case, we can use the trick from \cite[p.288-289]{Kushner92}).
Define $w^h\cd$ by
\beq{d.3x}
w^h(t)  = \sum\limits_{m=0}^{n^h(t)-1} \big[ 1/\sg (X^h_m, \al^h_m)\big]
(\Delta X^h_m -\E^{h}_m \Delta X^h_m).
\eeq
Then we can write
\beq{d.4}
M^h(t) =\int_0^t \sg (X^h(s), \al^h(s)) dw^h(s) + \e_2^h(t),
\eeq
with $\{\e_2^h\cd\}$
 being an $\mathcal{F}^h(t)$-adapted process satisfying \begin{equation*}\lim\limits_{h\to 0} \sup\limits_{t\in [0, T_0]}\E|\e_2^h(t)|=0 \quad \text{for any }0<T_0<\infty.\end{equation*}
Using  \eqref{d.3} and \eqref{d.4}, we can write \eqref{d.2} as
\beq{d.5}
X^h(t) = x + \int_0^t \[b (X^h(s), \al^h(s))+U^h(s)\]  ds +  \int_0^t \sg(X^h(s), \al^h(s)) dw^h(s) +\e^h(t),
\eeq
where $\e^h\cd$ is an $\mathcal{F}^h(t)$-adapted process satisfying \begin{equation*}\lim\limits_{h\to 0} \sup\limits_{t\in [0, T_0]}\E|\e^h(t)|=0 \quad \text{for any }0<T_0<\infty.\end{equation*}
The performance
function from \eqref{e.7} can be rewritten as
\beq{d.6}
J^h(x, \al, U^h\cd)= \E \int_0^{\infty} e^{-\delta s} p\big(X^h(s), \al^h(s), U^h(s)\big)ds.
\eeq

\section{Convergence}

The convergence of the algorithms is established via the weak convergence method.
To proceed,  let $D[0, \infty)$ denote the space of functions that are right continuous and have left-hand limits endowed with the Skorokhod topology. All the weak analysis will be on this space or its $k$-fold products $D^k[0, \infty)$ for appropriate $k$. We follow \cite[Section 4.6]{Kushner92} in order to introduce relaxed control representations, which we need in order to prove the weak convergence.

\begin{defn}{\rm
Let $\mathcal{B}(\mathcal{U}\times [0, \infty))$ be the $\sigma$-algebra of Borel subsets of $\mathcal{U}\times [0, \infty)$. An admissible relaxed control, which we will call
a \textit{relaxed control}, $m\cd$ is a measure on $\mathcal{B}(\mathcal{U}\times [0, \infty))$ such
that $$m(\mathcal{U}\times [0, t]) = t \quad\text{for all}\quad t\ge 0.$$ Given a relaxed control $m\cd$,
there is a probability measure  $m_t\cd$ defined on the $\sigma$-algebra  $\mathcal{B} (\mathcal{U})$ such that $m(du dt)=m_t(du)dt$. Let $\mathcal{R}(\mathcal{U}\times [0, \infty))$ denote the set of all
relaxed controls on $\mathcal{U}\times [0, \infty)$.

With the given probability
space, we say that $m\cd$ is an admissible relaxed (stochastic) control if (i) for each fixed $t\ge 0$, $m(t, \cdot)$ is a random variable taking values in $\mathcal{R}(\mathcal{U}\times [0, \infty))$, and for each fixed $\omega$, $m(\cdot, \omega)$ is a deterministic relaxed control; (ii) the function defined by $m(A\times [0, t])$ is $\mathcal{F}(t)$-adapted for any $A\in \mathcal{B}(\mathcal{U})$. As a result,
with probability one, there is a measure $m_t(\cdot, \omega)$
on the Borel $\sigma$-algebra $\mathcal{B}(\mathcal{U})$ such that
$m(dcdt) = m_t(dc)dt$. }
\end{defn}

\begin{rem}\label{rem:relax} {\rm
For a sequence  of controls $U^h=\{U^h_n: n\in  \mathbb{Z}_{\ge 0}\}$, we define a sequence of relaxed control equivalence as follows.
First, we set $m_{t}^h(du)=\delta_{U^h(t)}(du)$ for $t\ge 0$, where $\delta_{U^h(t)}\cd$ is the probability measure concentrated at $U^h(t)$. Then $m^h\cd$ is defined by $m^h(dudt)=m_t(du)dt$; that is,
$$m^h(B\times [0, t])=\int_0^t \Big(\int_B\delta_{U^h(s)}(du)\Big)ds, \quad B\in \mathcal{B}({\mathcal{U}})\quad \text{and}\quad t\ge 0.$$

Recall that $\mathcal{R}(\mathcal{U}\times [0, \infty))$ is the space of all
relaxed controls on $\mathcal{U}\times [0, \infty)$.
 Then $\mathcal{R}(\mathcal{U}\times [0, \infty))$
 can be metrized using the Prokhorov metric in the usual way as in \cite[pages 263--264]{Kushner92}. With the Prokhorov metric,  $\mathcal{R}(\mathcal{U}\times [0, \infty))$
 is a compact space. It follows that any sequence of relaxed controls has a convergent subsequence.
 Moreover, a sequence $(\eta_n)_{n\in \mathbb{N}}$ with $\eta_n\in \mathcal{R}(\mathcal{U}\times [0, \infty))$ converges to $\eta\in \mathcal{R}(\mathcal{U}\times [0, \infty))$ if and only if for any continuous functions with compact support $\Psi\cd$ on $\mathcal{U}\times [0, \infty)$ one has
$$\int_{\mathcal{U}\times [0, \infty)}
\Psi(u, s)\eta_n(du, ds) \to \int_{\mathcal{U}\times [0, \infty)}
\Psi(u,s)\eta(du, ds)$$
as $n\to\infty$.
Note that for a sequence  of ordinary controls $U^h=\{U^h_n:  n\in \mathbb{Z}_{\ge 0}\}$, the associated relaxed control  ${m}^h(dcdt)$ belongs to $\mathcal{R}(\mathcal{U}\times [0, \infty))$.
 Note also that the limits of the ``relaxed
 control representations'' of the ordinary controls might not be ordinary controls, but
 only relaxed controls. }
\end{rem}

With the notion of relaxed control given above, we can write \eqref{d.5} and \eqref{d.6} as
\beq{dd.1}
X^h(t) = x + \int_0^t \[b (X^h(s), \al^h(s))+U(s)\]  m^h_s(du)ds +  \int_0^t \sg(X^h(s), \al^h(s)) dw^h(s) +\e^h(t),
\eeq
\beq{dd.2}
J^h(x, \al, m^h\cd)= \E \int_0^{\infty} e^{-\delta s} p\big(X^h(s), \al^h(s), u\big)m_s^h(du)ds.
\eeq
The value function defined in \eqref{e.5} can be rewritten as
$$V(x, \al)=\sup\{ J(x, \al, m\cd): m\cd \,\,\, \text{ is an admissible relaxed control}\},
$$
where
$$J(x, \al, m\cd):=  \E_{x, \al}  \int_0^{\infty} e^{-\delta s} p\big(X(s), \al(s), u\big) m_s(du)ds.$$

\begin{lem}
The process $\{\al^h\cd\}$ converges weakly to $\alpha\cd$, which is a Markov chain with generator $Q=(q_{kl})$.
\end{lem}
\begin{proof}
The proof is similar to that of \cite[Theorem 3.1]{YZB03} and is therefore omitted.
\end{proof}

\begin{thm}\label{thm:thm} Suppose Assumption \ref{a:1} holds.
	Let the
	 chain $\{(X^h_n, \al^h_n) \}$ be constructed using the transition probabilities defined in \eqref{a.2},
	$\big(X^h\cd, \al^h\cd, w^h\cd\big)$ be the continuous-time interpolation defined in \eqref{d.1} and \eqref{d.3x}, $\{U^h_n\}$ be an admissible strategy and $m^h\cd$ be the relaxed control representation of $\{U^h_n\}$.
Then the following assertions hold.
\begin{itemize}
\item[\rm(a)] The family of processes
$H^h\cd =\big({X}^h\cd, {\al}^h\cd, m^h\cd, {w}^h\cd\big)$ is tight. 	
	As a result, it has a weakly convergent subsequence with limit $H\cd = \big({X}\cd, {\al}\cd, m\cd, {w}\cd\big).$
		\item[\rm(b)] Let $\mathcal{ F}(t)$ be the $\sigma$-algebra generated by $\big\{H(s): s\le t\big\}$.  Then $w\cd$  is a standard $\mathcal{ F}(t)$ adapted Brownian motion, $m\cd$ is an admissible control,   and
 \beq{dd.3}
		X(t) = x + \int_0^t \[b (X(s), \al(s))+u\]  m_s(du)ds +  \int_0^t \sg(X(s), \al(s)) dw(s), \quad t\ge 0.
		\eeq
	\end{itemize}
\end{thm}

\begin{proof}

	(a)	We use the tightness criteria in \cite[p. 47]{Kushner84}. Specifically,  a sufficient condition for tightness of a sequence of processes  $\zeta^h\cd$ with paths in $D^k[0, \infty)$ is that for any $T_0, \rho
		\in (0, \infty)$,
		\bea
		\ad  \E_t^h\big|\zeta^h(t+s)-\zeta^h(t)\big|^2\le \E^h_t \gamma(h, \rho) \quad \text{for all}\quad s\in [0, \rho], \quad t\le T_0,\\
		\ad \lim\limits_{\rho\to 0}\limsup\limits_{h\to 0} \E  \gamma(h, \rho) =0.
		\eea
		The tightness of $\{\al^h\cd\}$ is obvious by the preceding lemma. The process $\{m^h\cd\}$ is tight since its range space is compact.
				It is standard to show the tightness of $\{w^h\cd\}$ and $\{X^h\cd\}$ -- see \cite{Song06} for details.
	As a result, a subsequence of $H^h\cd=\big(X^h\cd, \al^h\cd, m^h\cd, w^h\cd\big)$ converges weakly to the limit $H\cd=\big(X\cd, \al\cd, m\cd, w\cd\big)$.

(b) 	For the rest of the proof, we assume the probability space is chosen as required by Skorokhod representation. Thus, with a slight abuse of notation, we assume that $H^h\cd$ converges to the limit 	$H\cd$
		with probability one via Skorokhod representation.
		
		To characterize ${w}\cd$, let $\wdt k, \wdt j$ be arbitrary positive integers. Pick $t>0$, $\rho>0$ and $\{t_k: k\le \wdt k\}$ such that $t_k\le t\le t+\rho$ for each $k$. Let $\phi_j\cd$ be real-valued continuous functions that are compactly supported on $\mathcal{U}\times [0, \infty)$ for any $j\le \wdt j$. Define $(\phi_j, m)_t := \int_0^t \int_{\mathcal{U}} \phi_j (u, s)m(duds)$.
		
		Let $\Psi\cd$ be a real-valued and continuous function of its arguments with compact support. By the definition of $w^h\cd$ in \eqref{d.3x}, $w^h\cd$ is an $\mathcal{F}^h(t)$-martingale. Thus, we have
	\beq{dd.4}
	\E\Psi\big(X^h(t_k), \al^h(t_k), w^h(t_k), (\phi_j, m^h)_{t_k}, j\le \wdt j, k\le \wdt k\big)\big[ {w}^h(t+\rho)-{w}^h(t)\big]=0,
	\eeq
	and
	\beq{dd.5}
	\E\Psi\big(X^h(t_k), \al^h(t_k), w^h(t_k), (\phi_j, m^h)_{t_k}, j\le \wdt j, k\le \wdt k\big)\big[ \big({w}^h(t+\rho)\big)^2-\big({w}^h(t)\big)^2-\rho-{\e}^h(\rho)\big]=0.
	\eeq
	By using the Skorokhod representation and the dominated convergence theorem, letting $h\to 0$ in \eqref{dd.4}, we obtain
	\beq{dd.6}
	\E\Psi\big(X(t_k), \al(t_k), w(t_k), (\phi_j, m)_{t_k}, j\le \wdt j, k\le \wdt k\big)\big[ {w}(t+\rho)-{w}(t)\big]=0.
	\eeq
	Since ${w}\cd$  has continuous paths with probability one,
	\eqref{dd.6} implies that ${w}\cd$ is a continuous ${\mathcal{F}}\cd$-martingale. Moreover, \eqref{dd.5} gives us that
	\beq{}
	\E\Psi\big(X(t_k), \al(t_k), w(t_k), (\phi_j, m)_{t_k}, j\le \wdt j, k\le \wdt k\big)\big[ \big({w}(t+\rho)\big)^2-\big({w}(t)\big)^2-\rho\big]=0.
	\eeq
	Thus, the quadratic variation of $w(t)$ is $t$, which implies that $w\cd$  is a standard $\mathcal{ F}(t)$ adapted Brownian motion.

By the convergence with probability one via Skorokhod representation, we have
	$$\E\left| \int_0^t \int_{\mathcal{U}} \left[b(X^h(s), \al^h(s)) + u\right]m_s^h(du)ds - \int_0^t \int_{\mathcal{U}} \left[b(X(s), \al(s)) + u\right]m_s^h(du)ds \right|\to 0$$
	uniformly in $t$ as $h\to 0$.
	
	 Also, by the weak convergence of $\{m^h\cd\}$, for any bounded and continuous function $\phi\cd$ with compact support, $(\phi, m^h)_\infty\to (\phi, m)_\infty$; see also Remark \ref{rem:relax}. The weak convergence and the Skorokhod representation imply that
	$$\int_0^t \int_{\mathcal{U}} \big[b(X(s), \al(s)) + u\big]m_s^h(du)ds - \int_0^t \int_{\mathcal{U}}\big[b(X(s), \al(s)) + u\big]m_s(du)ds\to 0$$
	uniformly in $t$ on any bounded interval with probability one.

For each positive constant $\rho$ and process ${\nu}\cd$, define the piecewise constant process ${\nu}^\rho\cd$
	by
	${\nu}^\rho(t)={\nu}(k\rho)$ for $t\in [k\rho, k\rho+\rho), k\in  \mathbb{Z}_{\ge 0}$. Then, by the tightness of
	$({X}^h\cd, {\al}^h\cd)$,
	\eqref{dd.1} can be rewritten as
	$$
	{X}^h(t) = x + \int_0^t \int_{\mathcal{U}} \big[b ({X}^h(s), {\al}^h(s)) +u\big]m_s^h(du)ds + \int_0^t \sg({X}^{h, \rho}(s), {\al}^{h, \rho}(s)) d {w}^h(s)  + {\e}^{h, \rho}(t),
	$$
	where
	$\lim\limits_{\rho\to 0}\limsup\limits_{h\to 0} \E|{\e}^{h, \rho}(t)|=0.$
	Noting that the processes
	${X}^{h, \rho}\cd$ and ${\al}^{h, \rho}\cd$
	take constant values
	on the intervals $[n\rho, n\rho+\rho)$, we have
	$$\int_0^t  \sg({X}^{h, \rho}(s), {\al}^{h, \rho}(s))d{w}^h(s)\to \int_0^t  \sg({X}^\rho(s), {\al}^\rho(s))d{w}(s) \quad \text{ as }\quad  h\to 0.
	$$
	The integrals above are well defined with probability one since they can be written as finite sums.
	Combining the last results, we have
	$$
	{X}(t)=x +\int_0^t \int_{\mathcal{U}} \big[ b({X}(s), {\al}(s))+u\big] m_s(du)ds+\int_0^t  \sg({X}^\rho(s), {\al}^\rho(t))d{w}(s)  +{\e}^{\rho}(t),
	$$
	where
	$\lim\limits_{\rho\to 0}E|{\e}^{ \rho}(t)|=0.$ Taking the limit as $\rho \to 0$ finishes the proof.
\end{proof}

\begin{thm}\label{thm:4.6} Suppose Assumption \ref{a:1} holds. Let $V^h(x, \al)$ and $V(x, \al)$ be the value functions defined in \eqref{e.5} and \eqref{e.7}. Then $V^h(x, \al)\to V(x, \al)$ as $h\to 0$.
\end{thm}

\begin{proof}
The proof is motivated by that of Theorem 7 in \cite{Song06}.
Let $U^h\cd$ be an admissible strategy for the chain $\{(X^h_n, \al^h_n)\}$ and  $m^h\cd$ be the corresponding relaxed control representation. Without loss of generality (passing to an additional subsequence if needed), we assume that  $ \big({X}^{{h}}\cd, {\al}^{{h}}\cd, {w}^{{h}}\cd, {m}^{{h}}\cd\big)$
	converges weakly to
	$\big({X}\cd, {\al}\cd, {w}\cd, {m}\cd\big)$.
	We show that as $h\to 0$ we have
	\beq{dd.7}
	J^h(x, \al, U^h\cd) \to
	J(x, \al,m\cd).
	\eeq
	From \eqref{dd.2} one has
	\beq{dd.8}
	\barray
	J^h (x, \al, U^h\cd) \ad
	=  \E \int_0^{\infty} e^{-\delta s} p\big(X^h(s), \al^h(s), u\big)m_s^h(du)ds.
	\earray
	\eeq
	By the weak convergence and the Skorokhod representation, as $h\to 0$,
	$$J^h (x, \al, U^h\cd) \to
	  \E \int_0^{\infty} e^{-\delta s} p\big(X(s), \al(s), u)\big)m_s(du)ds.$$
	  This yields that $J^h(x, \al, U^h\cd) \to
	J(x, \al,m\cd)$ as $h\to0$.

	Next, we prove that
	\beq{dd.9}
	\limsup\limits_{h\to 0} V^h(x, \al) \le V(x, \al).
	\eeq
	For any small positive constant $\e$, let $\wdt{U}^h\cd$ be an $\e$-optimal harvesting strategy for the chain $\{(X^h_n, \al^h_n)\}$; that is,
	$$V^h(x, \al)=\sup\limits_{U^h\cd} J^h(x, \al, U^h\cd)\le J^h(x, \al, \wdt{U}^h\cd) + \e.$$
	Choose a subsequence $\{\wdt{h}\}$ of $\{h\}$
	such that
	\beq{dd.10}\limsup\limits_{{h}\to 0} V^{{h}}(x, \al)=\lim\limits_{\wdt{h}\to 0}V^{\wdt{h}} (x, \al)\le\limsup\limits_{\wdt{h}\to 0} J^{\wdt{h}}(x, \al, {\wdt{U}}^{\wdt{h}}\cd)+\e.\eeq
Let $\wdt{m}^{\wdt{h}}\cd$ be the relaxed control representation of $\wdt{U}^{\wdt{h}}\cd$.	Without loss of generality (passing to an additional subsequence if needed), we may assume that
	$ \big({X}^{\wdt{h}}\cd, {\al}^{\wdt{h}}\cd, {w}^{\wdt{h}}\cd, {m}^{\wdt{h}}\cd\big)$
	converges weakly to
	$\big({X}\cd, {\al}\cd, {w}\cd, {m}\cd\big)$.
	It follows from our claim in the beginning of the proof that
	\beq{dd.11}
	\lim\limits_{\wdt{h}\to 0} J^{\wdt{h}}(x, \al,  {\wdt{U}}^{\wdt{h}}\cd)=\lim\limits_{\wdt{h}\to 0} J^{\wdt{h}}(x, \al,  {\wdt{m}}^{\wdt{h}}\cd)= J(x, \al,  m\cd )\le V(x, \al),
	\eeq
	where	$J(x, \al, m\cd)\le V(x, \al)$
	by the definition of $V(x, \al)$.
	Since $\e$ is arbitrarily small, \eqref{dd.9} follows from \eqref{dd.10} and \eqref{dd.11}.

	To prove the reverse inequality
	$\liminf\limits_{h} V^h(x, \al)\ge V(x, \al)
	$, for any small positive constant $\e$,
	we choose a particular $\e$-optimal strategy $\lbar m\cd$ for \eqref{e.2}-\eqref{e.3} such that
	the approximation can be applied
	to the chain $\{(X^h_n, \al^h_n)\}$
	and the associated cost compared with $V^h(x, \al)$. By the chattering lemma (see for instance \cite[Theorem 3.1]{Kushner90}), for any given $\e>0$, there is a constant $\lambda>0$ and an ordinary control
	$\lbar U^\e\cd$
	for \eqref{e.2}-\eqref{e.3} with the following properties:
\begin{enumerate}[(a)]
  \item $\lbar U^\e\cd$ takes only finitely many values (denoted by $\mathcal{U}_\e$ the set of all such values);
  \item $\lbar U^\e\cd$ is constant on the intervals $[k\lambda, k\lambda + \lambda)$ for $k\in  \mathbb{Z}_{\ge 0};$
  \item with $\lbar m^\e\cd$ denoting the relaxed control representation of $\lbar U^\e\cd$, we have that \newline$(\lbar X^\e\cd, \lbar \al^\e\cd, \lbar w^\e\cd, \lbar m^\e\cd)$ converges weakly to
	$(\lbar X\cd, \lbar \al\cd, \lbar w\cd, \lbar m\cd)$ as $\e\to 0$;
\item $J(x, \al, \lbar m^\e\cd)\ge V(x, \al) -\e$.
\end{enumerate}
For $\e>0$ and the corresponding $\lambda$ in the chattering lemma, consider an optimal control problem for \eqref{e.2} subject to \eqref{e.3}, but where the controls are constants over the interval $[k\lambda, k\lambda +\lambda)$ for $k\in  \mathbb{Z}_{\ge 0}$ and take values in $\mathcal{U}_\e$ (the set of control values of $\lbar U^\e\cd$). This corresponds to controlling the discrete-time Markov process that is obtained by sampling $X\cd$ and $\al\cd$ at times $k\lambda$ for $k\in  \mathbb{Z}_{\ge 0}$. Let $\wdh{U}^\e\cd$ denote the $\e$-optimal control, $\wdh{m}^\e\cd$ denote the relaxed control representation, and let $\wdh{X}^\e\cd$ denote the associated state process. Since $\wdh{m}^\e\cd$ is $\e$-optimal in the chosen class of controls, we have
$$J(x, \al, \wdh{m}^\e\cd)\ge J(x, \al, \lbar m^\e\cd)-\e\ge V(x, \al)-2\e.$$
We next approximate $\wdh{U}^\e\cd$ by a suitable function of $w\cd$ and $\al\cd$.
Using the same method as in \cite{Song06}, we can approximate $\wdh{U}^\e\cd$ by
the ordinary control $U^{\e, \theta}\cd$ with the corresponding relaxed control $m^{\e, \theta}\cd$ and the state process $X^{\e, \theta}\cd$ such that $$m^{\e, \theta}\cd\to \wdh{m}^\e \cd$$ as $\theta\to 0$ and $$J(x, \al, m^{\e, \theta}\cd)\ge J(x, \al, \wdh{m}^\e\cd) - \e\ge V(x, \al)-3\e.$$ Then a sequence of ordinary controls $\{\lbar U^h_n\}$ for the chain $\{(X^h\cd, \al^h\cd)\}$ can be constructed with the relaxed control representation $\{\lbar m^h_n\}$  such that as $h\to 0$, the $(X^h\cd, \al^h\cd, \lbar m^h\cd, w^h\cd)$ converges weakly to $(X^{\e, \theta}\cd, \al\cd, m^{\e, \theta}\cd, w\cd)$. By the optimality of $V^h(x, \al)$ and the weak convergence above, we have as $h\to 0$,
$$V^h(x, \al) \ge J(x, \al, \lbar m^h\cd)\to J(x, \al, m^{\e, \theta}\cd).$$ It follows that
$V^h(x, \al)\ge V(x, \al)-4\e$ for sufficiently small $h$.
Since any subsequence of $H^h\cd$ has a subsequence that converges weakly and $\e
$ is arbitrary, we have $\liminf\limits_{h}V^h(x, \al)\ge V(x, \al)$. The conclusion follows.
\end{proof}

\section{Numerical Experiments}\label{sec:numerical}
\subsection{Varying the cost dependency}
\

\noindent
We want to see what effect different specifications of the cost function have on the shape of the optimal harvesting rate, and in particular whether it is bang-bang. We suspect that the convexity of the cost function leads to bang-bang (all or nothing) optimal harvesting. In Figure \ref{fig:cost_exp_1}, we have as an example a cost functions of the form $C(u) = \sqrt{|u|})$. The rest of the parameters are kept the same as in Subsection \ref{sec:intro_lv}. This example has concave costs, but there is a point of convexity at 0. Experiments with other partly concave cost functions show a similar pattern. Piecewise linear costs like $C(u)= |u|$, or $C(u) = \ln(1+|u|)$, lead to optimal controls that are step functions. However, when we use a purely concave cost function, like $C(u) = \ln(1+u/3)$, seen in Figure \ref{fig:cost_exp_3}, we again obtain bang-bang optimal control. Further experiments confirm the observation.
\begin{figure}[h!tb]
	\begin{center}
		\subfloat{{
			\includegraphics[scale=0.5]{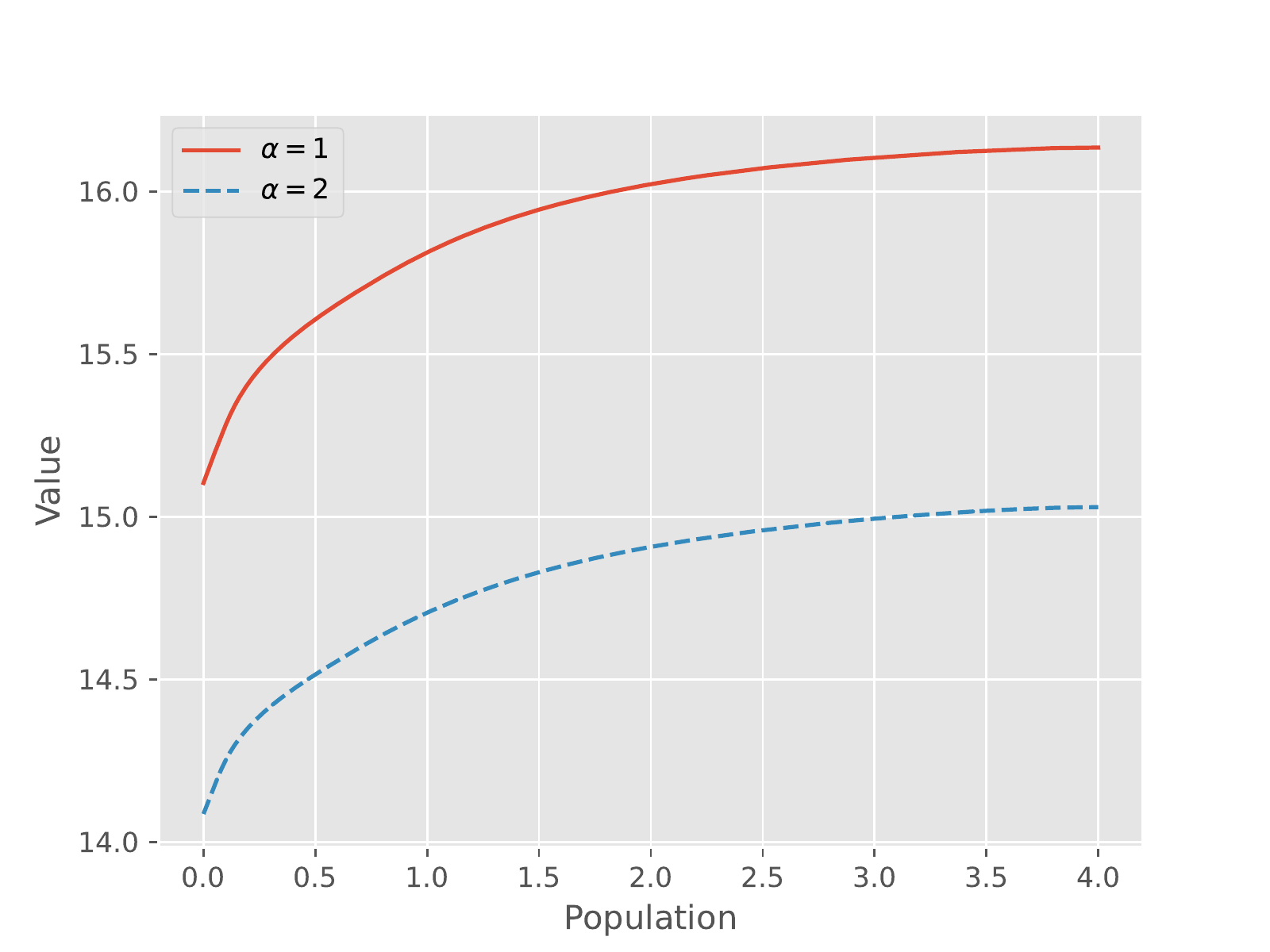}
		}}
		\subfloat{{
			\includegraphics[scale=0.5]{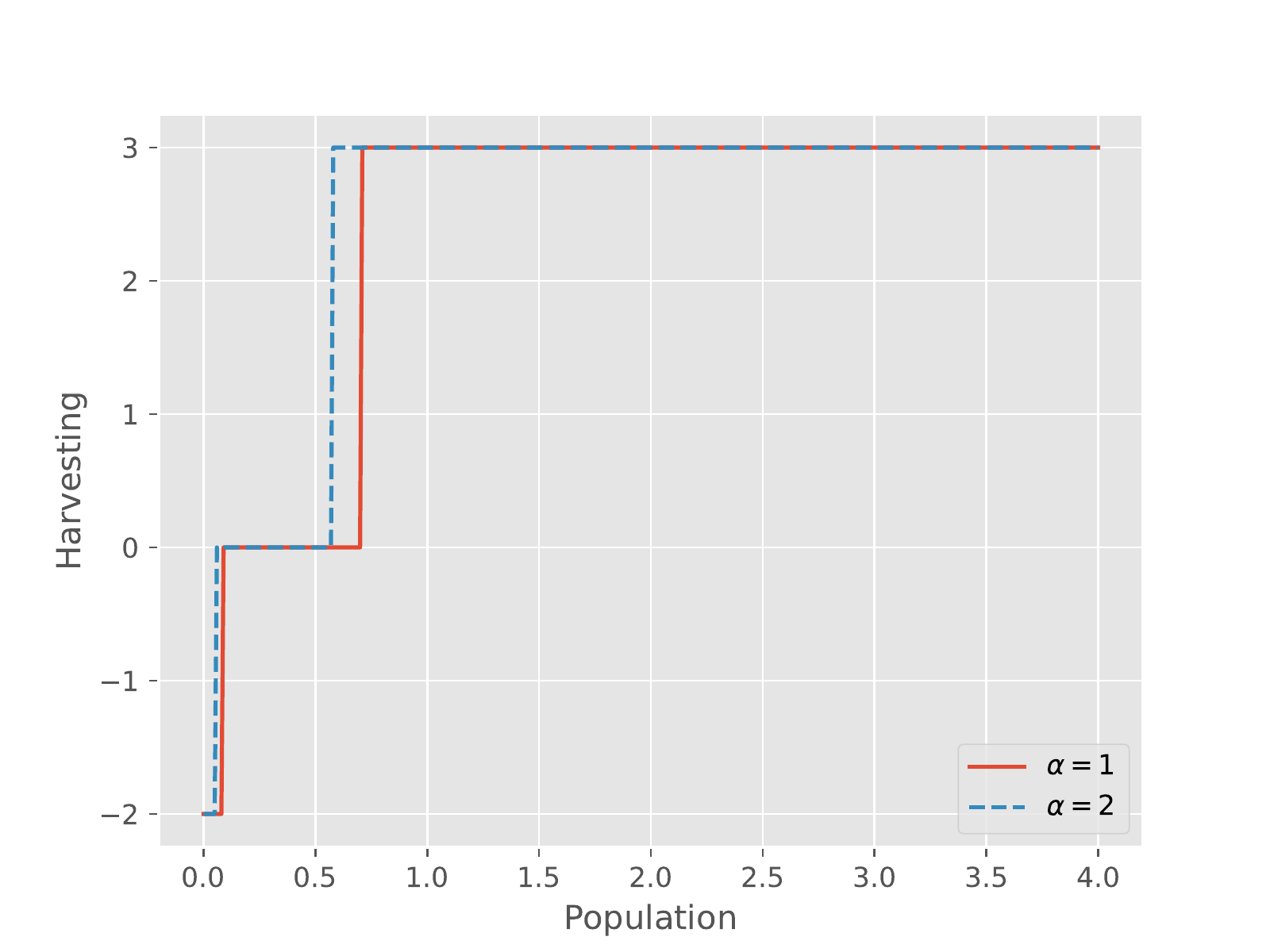}
		}}
		\caption{Value function (left) and optimal harvesting-stocking rate (right) for a model with switching affecting $\mu(\al) = 4 - \al$, and a cost function $C(u) = \sqrt{|u|}$. Other parameters described in Subsection \ref{sec:intro_lv}.} \label{fig:cost_exp_1}
	\end{center}
\end{figure}


\begin{figure}[h!tb]
	\begin{center}
		\subfloat{{
			\includegraphics[scale=0.5]{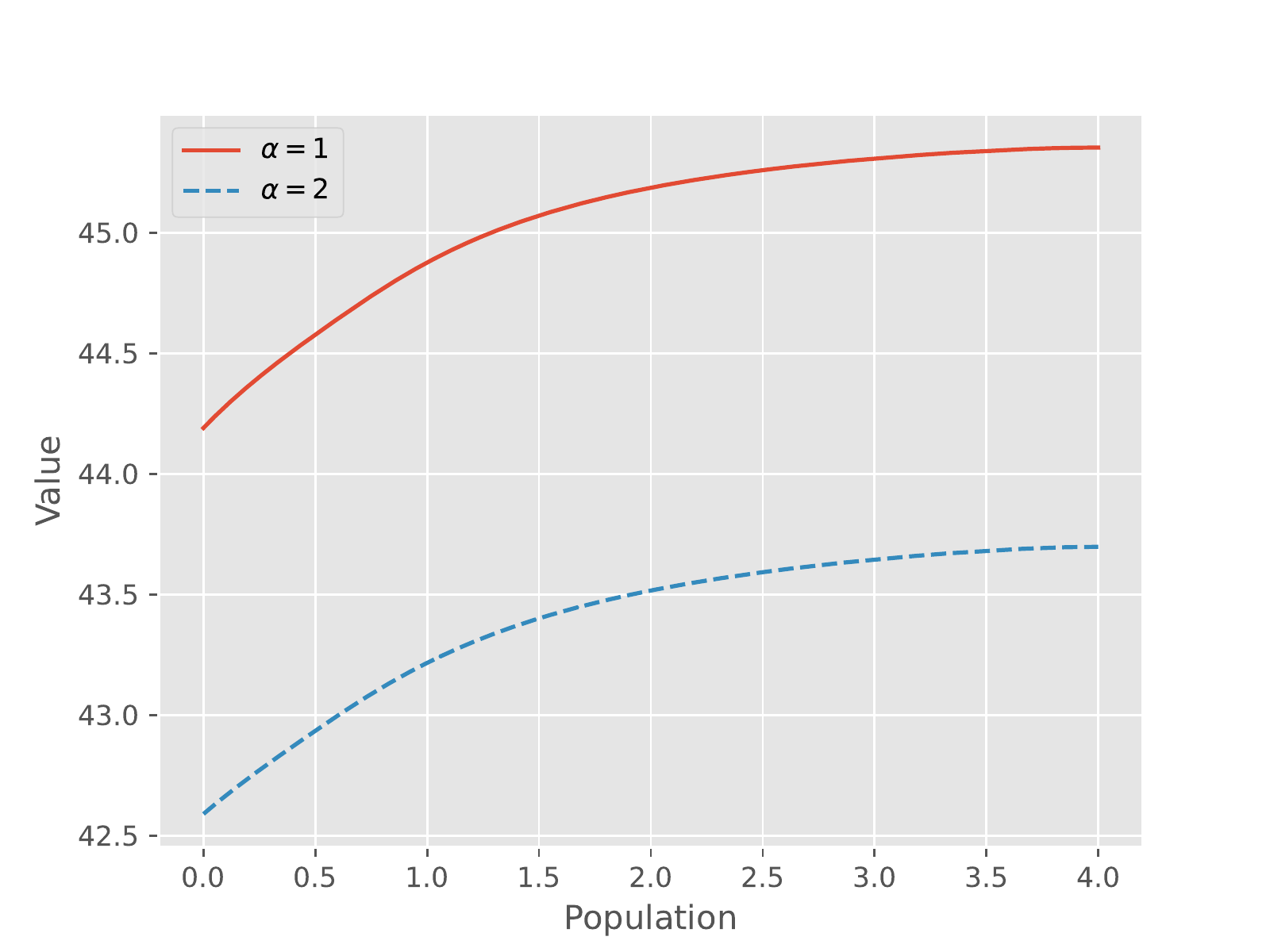}
		}}
		\subfloat{{
			\includegraphics[scale=0.5]{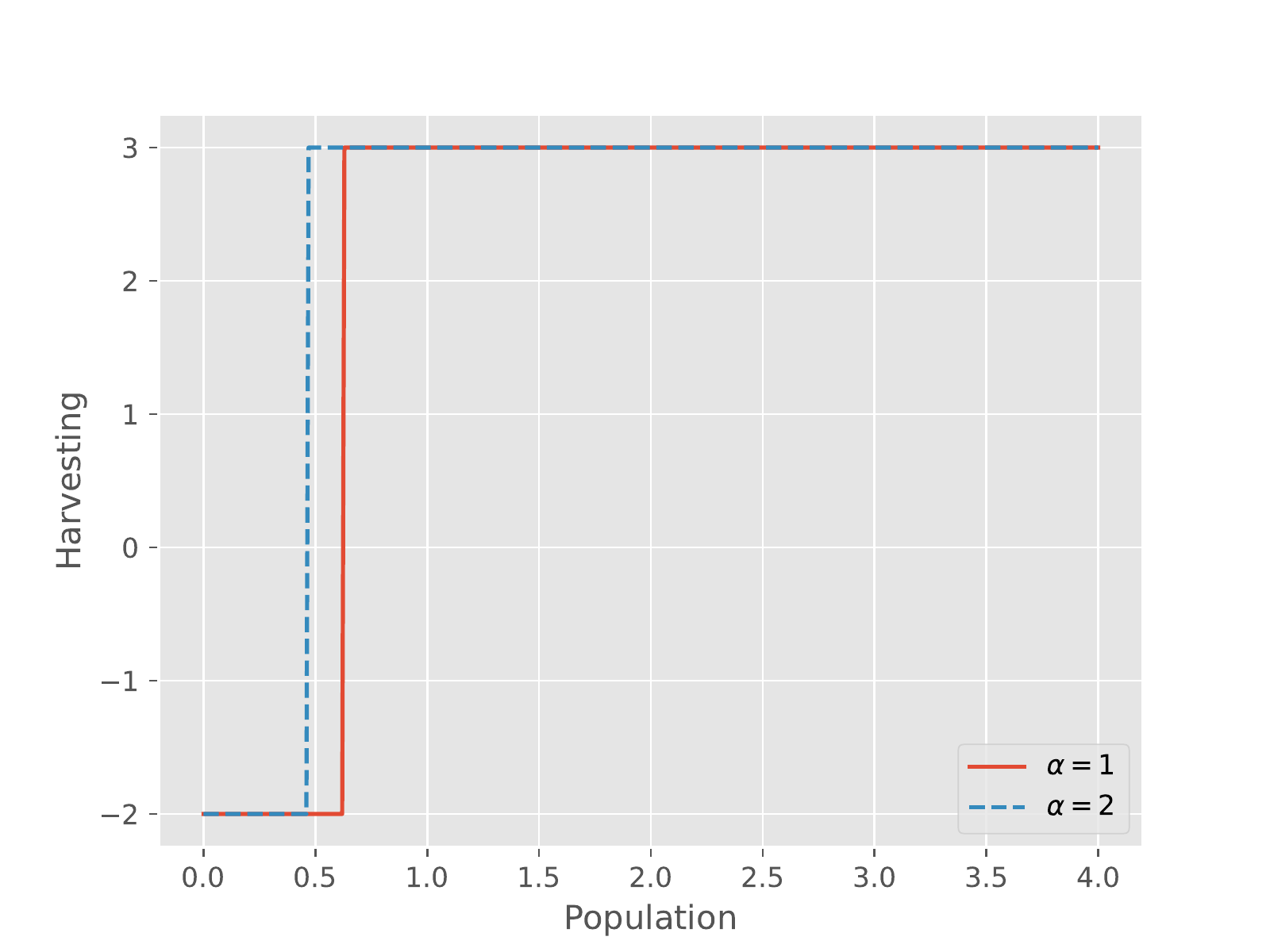}
		}}
		\caption{Value function (left) and optimal harvesting-stocking rate (right) for a model with switching affecting $\mu(\al) = 4 - \al$, and a cost function $C(u) = \ln(1+u/3)$. Other parameters described in Subsection \ref{sec:intro_lv}.} \label{fig:cost_exp_3}
	\end{center}
\end{figure}

\subsection{The Gompertz model of population growth}
\

\noindent
In this example, the dynamics of the population size without harvesting is given by a Gompertz model \citep{W32,Z93} of the form
\begin{equation*}
	dX(t)=\big[b(X(t), \alpha(t)) -U(t)X(t)\big]dt + \sigma(X(t), \alpha(t))dw(t),
\end{equation*}
where
\begin{align*}
	& b(x, \alpha)=(4 - \alpha) x \ln \dfrac{2}{x}, \quad \sigma(x, \alpha) = x, \\
	& \mathcal{U}=\{u :u = k/500, k\in \mathbb{Z}, -1000\le k\le 1500\}, \quad (x, \alpha)\in  \mathbb{R}_+ \times \{1, 2\}.
\end{align*}

The generator $Q$ of the Markov chain $\al\cd$ is given by
\begin{equation*}
	q_{11}=-0.1, \quad q_{12}=0.1, \quad q_{21}=0.1, \quad q_{22}=-0.1.
\end{equation*}

\begin{figure}[h!tb]
	\begin{center}
		\subfloat{{
			\includegraphics[scale=0.5]{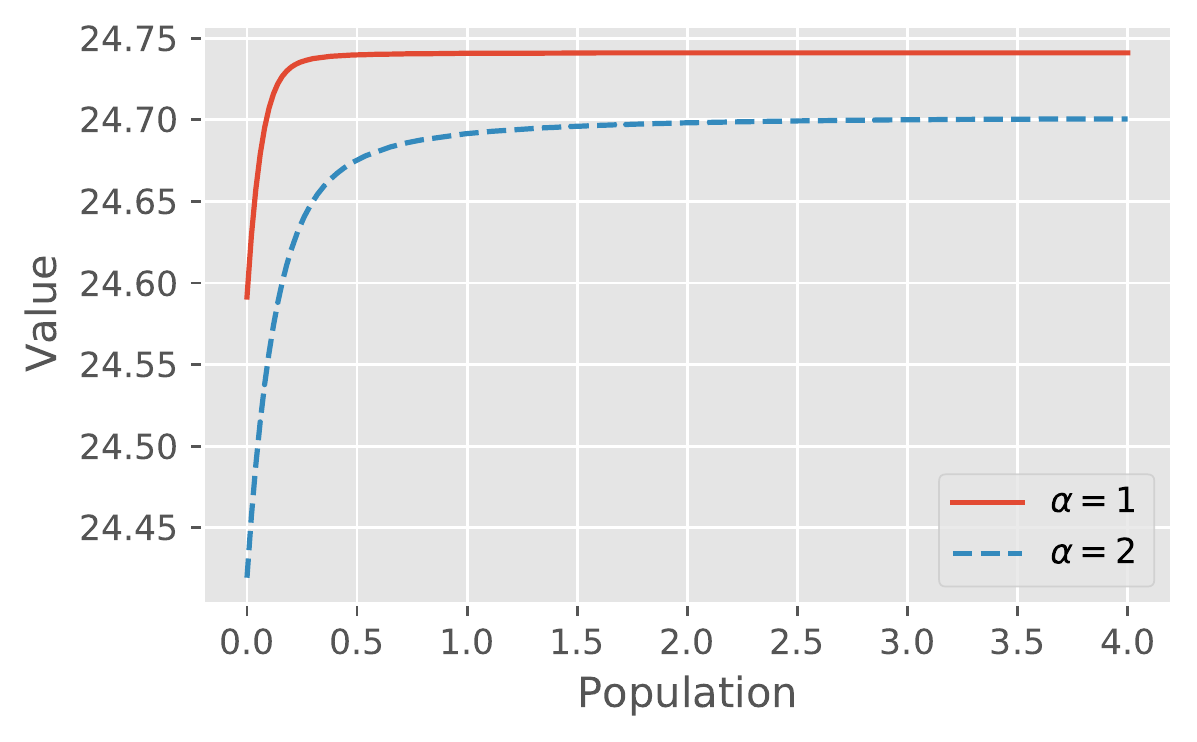}
		}}
		\subfloat{{
			\includegraphics[scale=0.5]{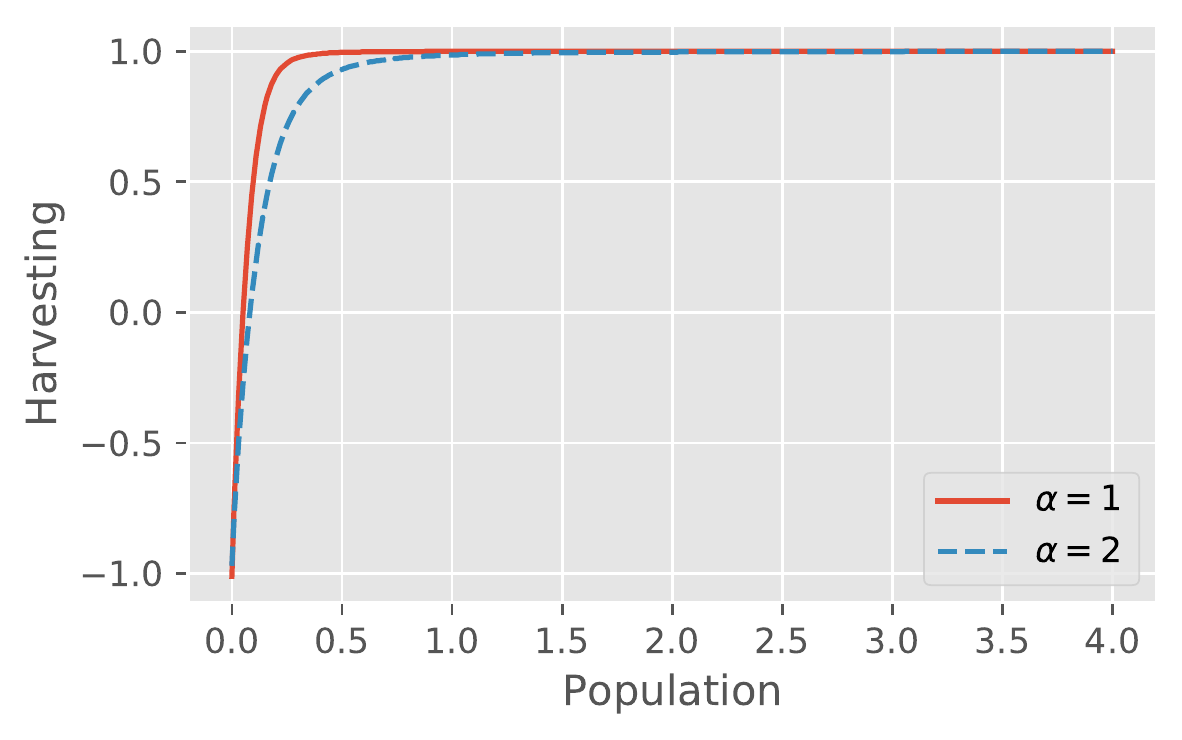}
		}}
		\caption{Value function (left) and optimal harvesting-stocking rate (right) for a Gompertz model with absolute harvesting, with switching affectin $b(x, \alpha)=(4 - \alpha) x \ln \frac{2}{x}$, constant price $P(\cdot) = 1$, and a cost function $C(\cdot) = u^2/2$. Other parameters described in Subsection \ref{sec:intro_lv}.} \label{fig:Gompertz}
	\end{center}
\end{figure}
Figure \ref{fig:Gompertz} shows the value function and the optimal stocking-harvesting rate as a function of population size $X(t)$ and the environmental state $\al$.
In the Gompertz model, the deterministic rate of growth near extinction goes to $\infty$, unlike in the logistic model where it is linear.
Comparing these results with the ones in Figure \ref{fig:switching_cost_1}, we can also see that low population values in the Gompertz model are much less unfavorable, both in terms of future value and in terms of the benefit of extraction.

\subsection{The Nisbet-Gurney model of population growth}
\

\noindent
In this model, the evolution of the population size without harvesting is given by a switched Nisbet-Gurney model; thus,
\begin{equation*}
	dX(t)=\big[b(X(t), \alpha(t)) -U(t)\big]dt + \sigma(X(t), \alpha(t))dw(t),
\end{equation*}
where
\begin{align*}
	& b(x, \alpha) = (4-\alpha) x e^{-x} - x, \quad \sigma(x, \alpha)= x, \\
	& \mathcal{U}=\{u :u = k/500, k\in \mathbb{Z}, -1000\le k\le 1500\}, \quad (x, \al)\in  \mathbb{R}_+ \times \{1, 2\}.
\end{align*}
The generator $Q$ of the Markov chain $\alpha(\cdot)$ is given by
\begin{equation*}
	q_{11}=-0.1, \quad q_{12}=0.1, \quad q_{21}=0.1, \quad q_{22}=-0.1.
\end{equation*}

\begin{figure}[h!tb]
	\begin{center}
		\subfloat{{
			\includegraphics[scale=0.5]{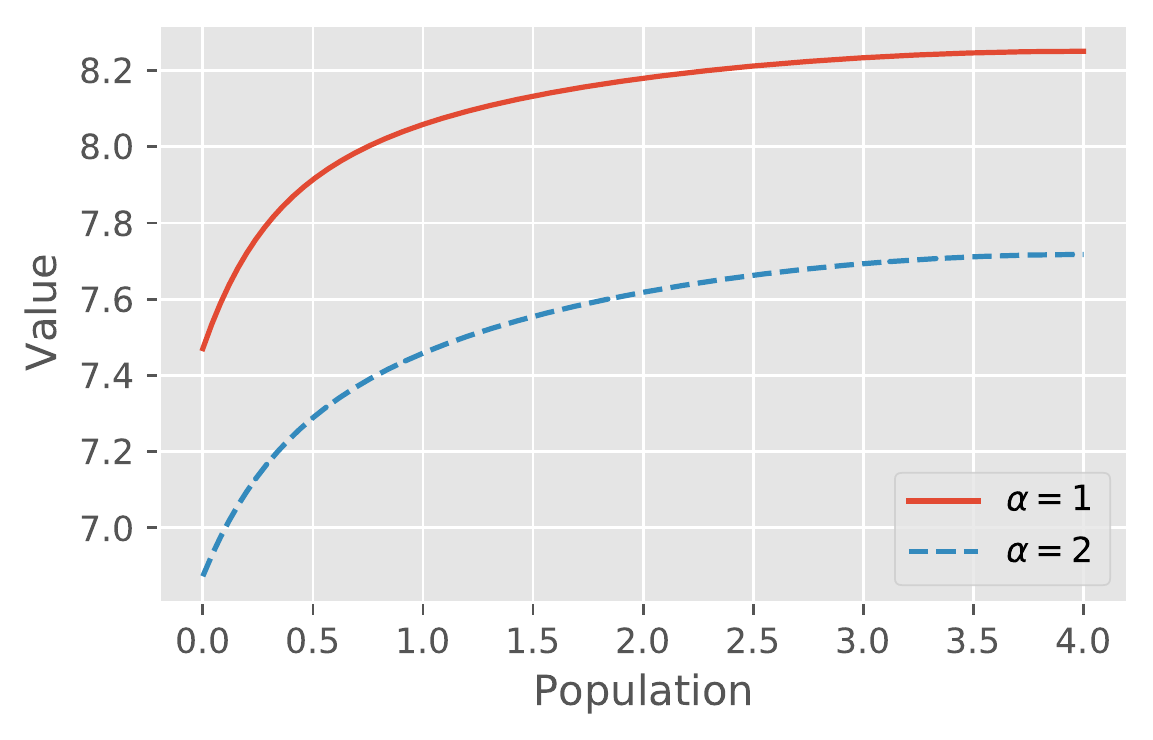}
		}}
		\subfloat{{
			\includegraphics[scale=0.5]{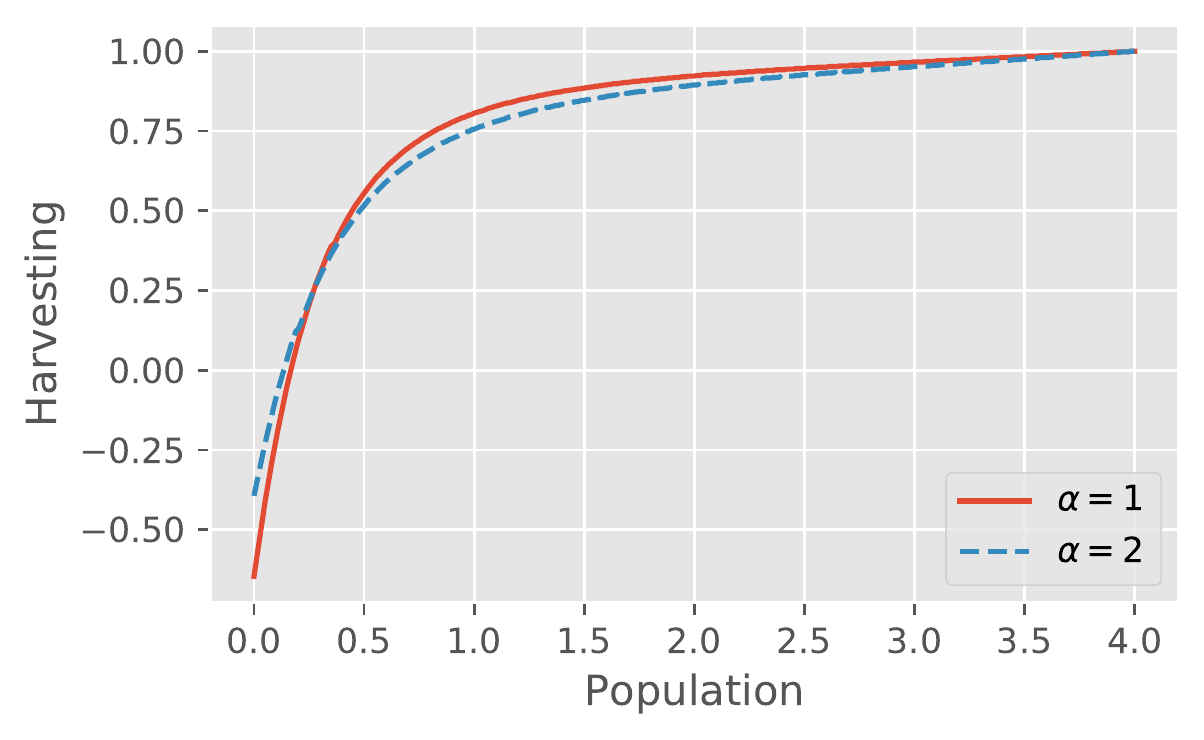}
		}}
		\caption{Value function (left) and optimal harvesting-stocking rate (right) for a Nisbet-Gurney model with absolute harvesting, with switching affecting the growth rate $b(x, \alpha) = (4-\alpha) x e^{-x} - x$, constant price $P(\cdot) = 1$, and a cost function $C(\cdot) = u^2/2$. Other parameters described in Subsection \ref{sec:intro_lv}.} \label{fig:NisGur}
	\end{center}
\end{figure}

Figure \ref{fig:NisGur} shows a numerical estimation of this model. The value function has the usual features, being increasing and concave.
The harvesting rate is monotonic, which is not a surprise considering the cost choice and our discussion in Section \ref{sec:var_eff}.
Again, the control in state $\alpha = 1$ shows higher harvesting and seeding, which is consistent with this state being more favourable for growth.


\end{document}